\documentclass[11pt]{article}
\usepackage[french,english]{babel}
\usepackage[latin1]{inputenc}
\usepackage{amsmath}
\usepackage{amsfonts}
\usepackage{amssymb}
\usepackage{amsthm}
\usepackage{mathrsfs}
\usepackage{enumitem}

\newtheorem{thm}{Theorem}

\newtheorem{rem}[thm]{Remark}
\newtheorem{cor}[thm]{Corollary}
\newtheorem{prop}[thm]{Proposition}

\renewcommand{\geq}{\geqslant}
\renewcommand{\leq}{\leqslant}

\begin{document}

\title{{\bf Green functions
% and Martin boundary 
for killed random walks\\in the Weyl chamber of Sp{\boldmath$(4)$}}
}

\author{{\large Kilian Raschel\footnote{Laboratoire de
         Probabilit\'es et Mod\`eles Al\'eatoires, Universit\'e
         Pierre et Marie Curie, 4 Place Jussieu, 75252 Paris Cedex 05,
         France. E-mail: \texttt{kilian.raschel@upmc.fr}}}}

\date{\today}

\maketitle

\begin{abstract}
     We consider a family of random walks killed at the boundary
     of the Weyl chamber of the dual of $\rm{Sp}(4)$, which in
     addition satisfies the following property: for any $n\geq 3$,
     there is in this family a walk associated with a reflection
     group of order $2n$. Moreover, the case $n=4$ corresponds to
     a process which appears naturally by studying quantum random
     walks on the dual of $\rm{Sp}(4)$. For all the processes 
     belonging to this family, we find the exact asymptotic of 
     the Green functions along all infinite paths of states as 
     well as that of the absorption probabilities along the boundaries.
\end{abstract}
\selectlanguage{francais}
\begin{abstract}
Dans cet article, nous considérons une famille de marches aléatoires tuées au bord de
la chambre de Weyl du dual de $\rm{Sp}(4)$, qui vérifie en outre
la propriété suivante: pour tout $n\geq3$, il y a, dans cette famille,
une marche ayant un groupe de réflexions d'ordre $2n$. De plus, le
cas $n=4$ correspond à un processus  bien connu apparaissant 
lors de l'étude des marches aléatoires quantiques sur le dual de $\rm{Sp}(4)$. 
Pour tous les processus de cette famille, nous trouvons l'asymptotique
exacte des fonctions de Green selon toutes les trajectoires, ainsi
que l'asymptotique des probabilités d'absorption sur le bord.  
\end{abstract}
\selectlanguage{english}


\vspace{1mm}

\noindent {\it AMS $2000$ Subject Classification: primary 60G50, 31C35;
secondary 30E20, 30F10}

\section{Introduction and main results}
\label{sp4-Intro}

Appearing in several distinct domains, random walks conditioned on
staying in cones of $\mathbb{Z}^{d}$ attract more and more attention
from the mathematical community. Historically, important examples
are  the so-called non-colliding random walks. These
are the processes $(Z_{1}, \ldots, Z_{d})$ composed of $d$
independent and identically distributed random walks conditioned on
never leaving the Weyl chamber $\{z\in \mathbb{R}^{d} : z_{1}<\cdots
<z_{d} \}$. They first appeared in the eigenvalues description of
important matrix-valued stochastic processes, see~\cite{Dy62}, and
are recently again very much studied, see~\cite{KO,EJP2010-11,KoS} and the
references therein. Another important area where  processes
conditioned on never leaving cones of $\mathbb{Z}^{d}$
appear is that of quantum random walks, see e.g.\ \cite{Bi1,Bi3}.

A usual way to condition random processes on staying in cones consists in
using Doob $h$-transforms.
These are functions which are harmonic, positive
inside of the cone and equal to zero on its boundary---or equivalently
harmonic and positive for the underlying killed processes.
%then the Doob $h$-transform of $Z$ will obviously never hit the
%boundary of the cone. Moreover, such a function $h$ is very
%naturally positive harmonic for the process $Z$ killed at the
%boundary of the cone.
It is therefore natural to be interested in finding all positive
harmonic functions for processes in cones of ${\mathbb Z}^d$ killed at
the boundary, and more generally to compute the Martin
compactification of such processes, that can e.g.\ be obtained from
the exact asymptotic of the Green functions.

We briefly recall \cite{Dynkin} that for a transient Markov chain with state space $E$, the {Martin compactification} of $E$ is the
smallest compactification $\widehat{E}$ of $E$ for which the Martin kernels $y\mapsto k_{y}^{x}=
G_{y}^{x}/G_{y}^{x_0}$ extend continuously---by $G_{y}^{x}$ we mean the {Green functions}
%of the process
%\textit{i.e.}\ the mean number of visits made by the process at $y$ starting from $x$,
and we denote\vspace{-0.8mm} by $x_0$ a reference state. $\widehat{E}\setminus E$ is usually called the {full Martin boundary}.
For  $\alpha\in\widehat{E}$, $x\mapsto k_{\alpha}^{x}$ is clearly superharmonic; then $\partial_{m}E=\{\alpha \in \widehat{E}\setminus E :
x\mapsto k_{\alpha}^{x}\text{ is minimal harmonic}\}$ is called the
{minimal Martin boundary}---a harmonic function $h$ is said minimal if $0\leq \widetilde{h}\leq h$
with $\widetilde{h}$ harmonic implies $\widetilde{h}=c h$ for some constant $c$. Then, every superharmonic
(resp.\ harmonic) function $h$ can be written as $h(x)=\int_{\widehat{E}}k_{y}^{x}\mu(\text{d}y)$
(resp.\ $h(x)=\int_{\partial_{m}E}k_{y}^{x}\mu(\text{d}y)$), where $\mu$ is some finite
measure, uniquely characterized in the second case above.

\medskip

In this context, the case of walks in cones of $\mathbb{Z}^d$
spatially homogeneous in the interior, with \emph{non-zero drift} and
killed at the boundary has held a great and fruitful deal of
attention.

%This is how that in~\cite{Bi3}, by studying quantum random walks in
%the Weyl chamber of the dual of Lie groups, Biane is naturally
%led to consider the case of SU$(3)$ and Sp$(4)$ and, this way,
%the classical random walks with non-zero drift, having transition
%probabilities inside of the Weyl chamber as represented on Figure
%\ref{sp4-Lie} but killed at the boundary of the Weyl chamber
%--~it is actually possible to obtain classical random walks from
%quantum random~walks by restricting the latter ones, initially
%defined on non-commutative von Neumann algebras, to commutative
%subalgebras, see e.g.\ \cite{Bi3}.
%the case of the spatially homogeneous random walks in the
%same Weyl chambers and with the same non-zero jump probabilities as
%previously, but now these jump probabilities are admitted not to be all
%equal to $1/3$ or to $1/4$, so that the mean drift vector may have non-zero coordinates.

For random walks on weight lattices in Weyl chambers of Lie groups,
Biane \cite{Bi3} finds the minimal
Martin boundary
 thanks to Choquet-Deny theory.
   Collins \cite{COL} obtains their Martin compactification.

In \cite{KR} we give a more detailed analysis for a certain class of walks in
dimension $d=2$.
  These are the random walks killed at the boundary of $\mathbb{Z}_{+}^{2}$, with
non-zero jump probabilities to
the eight nearest neighbors and having in addition a positive mean drift. The
asymptotic of the Green functions along all infinite paths of
states as well as that  of the  probabilities of absorption
along the axes are computed for the walks in this class. However, the
methods~of complex analysis used in \cite{KR} apply in dimension $d=2$ only.

   Ignatiouk-Robert \cite{II,III},  then Ignatiouk-Robert and Loree \cite{IL09}
   find the Martin compac- tification of the random walks
 in $\mathbb{Z}_+\times \mathbb{Z}^{d-1}$ and $\mathbb{Z}_{+}^{d}$
  ($d\geq 2$), with non-zero drift and killed
 at the boundary. They make very general assumptions on the
 jump probabilities.
   The~approach used there,  based on large deviation
   techniques and Harnack inequalities, seems not to be powerful for studying
 the asymptotic of the Green functions.
     Furthermore, having a~non-zero drift is an
essential hypothesis in \cite{II,IL09,III}. Last but not least, the
results of \cite{III} in the case $d \geq 3$ are conditioned by the
fact of been able: ``to identify the positive harmonic functions of a
random walk on $\mathbb{Z}^{d}$ which has zero mean and is killed at
the first exit from $\mathbb{Z}_{+}^{d}$; unfortunately, for $d \geq
2$, there are no general results in this domain'' (see~page~5
of~\cite{III}).

\medskip

As may this open problem suggest, the results and
methods dealing with  the asymptotic of Green functions or even with
the Martin compactifiction for random walks in domains of
$\mathbb{Z}^{d}$ with \emph{drift zero} and killed at the boundary
are actually scarce, even for $d=2$.

  In \cite{LL2} (resp.\ \cite{Lawlerloop}), approximations of the Green functions for simple
  random walks killed at the boundary of balls of $\mathbb{Z}^{d}$ (resp.\ of certain more general
  sets of $\mathbb{Z}^{2}$)
are computed by comparison with Brownian motion.
      Namely, the Green
functions $G(B)_{y}^{x}$ of Brownian motion (resp.\ %
random walk) killed at the boundary $B$ are related to the potential kernels
(or the Green functions if $d\geq 3$) $a_{y}^{x}$ of the
non-killed Brownian motion (resp.\ random walk), e.g.\ via the
``balayage formula'', see Chapter~4 of \cite{LL2}. These identities
have the same form for Brownian motion and random walk \cite{LL2}:
     \begin{equation*}
     \label{bmrw}
          G(B)_{y}^{x}=-a_{y}^{x}+\mathbb{E}_{x}\big[a_{y}^{S_{\tau_{B}}}\big]+F(B)^{x},
     \end{equation*}
where $S$ is the process, $\tau_B$ is the hitting time of the boundary $B$ and where the rest
$F(B)^x$ can be expressed in terms of $\tau_B$ and $x$, indeed see Chapter 4 of
\cite{LL2}. They are then compared term by term.~For the comparison of
potential kernels $a_y^x$
of the non-killed processes, classical formulas may be used, like that stated in
Chapter~4 of \cite{LL2}; for the comparison of positions of these
processes at time of absorption, strong approximations of random
walks by Brownian motion \cite{Komlos1,Komlos2} as well as Beurling
estimates \cite{Kestenn,LL} are usually exploited.
%This approach have been
%applied as well for obtaining approximations of Green functions of a
%two-dimensional simple random walk in a ball \cite{LL2}
%or in some
%more general sets \cite{Lawlerloop}.
  %To provide accurate enough estimations of the Green functions
%  for random walks in cones,

This approach in particular requires  rather precise estimates of the hitting time
$\tau_B$.
  It seems therefore difficult to use  for models of walks
   in cones where $\tau_B$ is painful to analyze. This is for example the case
    of the random walks in the half-plane considered in \cite{EJP2010-37,UchiyamaPreprint}
     and of the random walks in the quarter-plane
       we shall study in this paper
     (for which Chapter~F of \cite{mythesis} illustrates the complexity of hitting
     times).
   Remark \ref{rwMB} in Subsection~\ref{sub_harmo_bm} specifies~other reasons why this method
   via comparison with Brownian motion seems not to lead to enough
     satisfactory results
     in the analysis of the random walks
  we shall consider here.

For random walks in the half-plane, Uchiyama
\cite{EJP2010-37,UchiyamaPreprint} finds the asymptotic of the Green
functions from their trigonometric representations.

If the domain is a {quarter-plane},  the simplest case is the
cartesian product of two killed one-dimensional simple random walks with
mean zero. The Martin boundary then happens to be trivial \cite{PW}. 
Moreover, the exact asymptotic of the Green functions for these
processes is computed in Chapter~D of \cite{mythesis}.

From a Lie group theory point of view, the previous case corresponds
to the group product SU$(2)\times $SU$(2)$, which is associated with
a \emph{reducible} rank-$2$ root system, see~\cite{Bou2}. We are
then interested in the classical random walks that can be obtained
from the construction made by  Biane in~\cite{Bi3}---namely, by
restriction of the quantum walks, initially
defined on non-commutative von Neumann algebras, to commutative
subalgebras---starting from  Lie groups
associated with \emph{irreducible} rank-$2$ roots systems, namely
SU$(3)$ and Sp$(4)$.
   These are random walks on the lattices of Figure~\ref{sp4-Lie},
  spatially homogeneous in the interior and absorbed at the boundary.
  The one on the left, associated with SU$(3)$, has three jump probabilities
  equal to $1/3$;
  that on the right, related to Sp$(4)$, has the same four jump probabilities $1/4$.

\begin{figure}[!ht]
 \begin{picture}(10.00,60.00)
 \hspace{20mm}
 \includegraphics{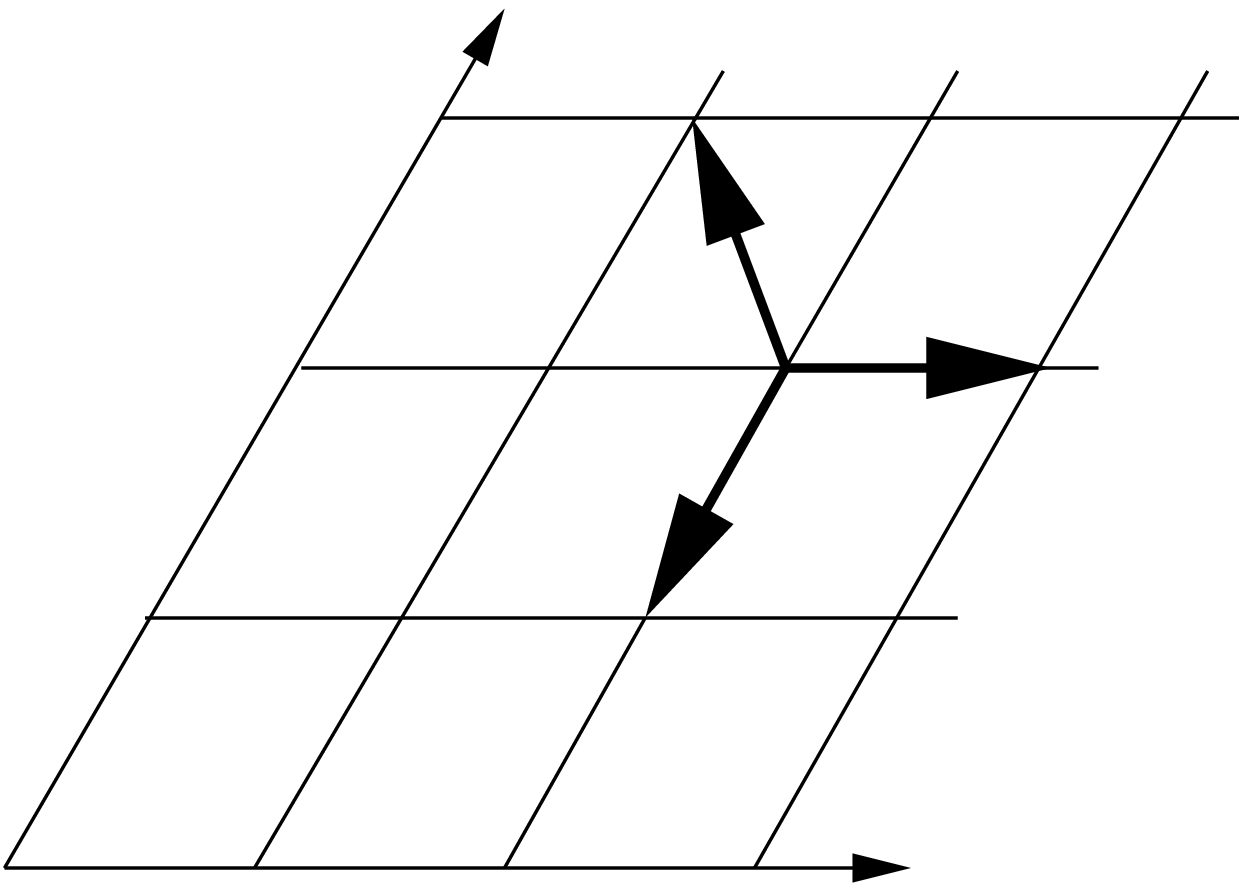}
 \hspace{70mm}
 \includegraphics{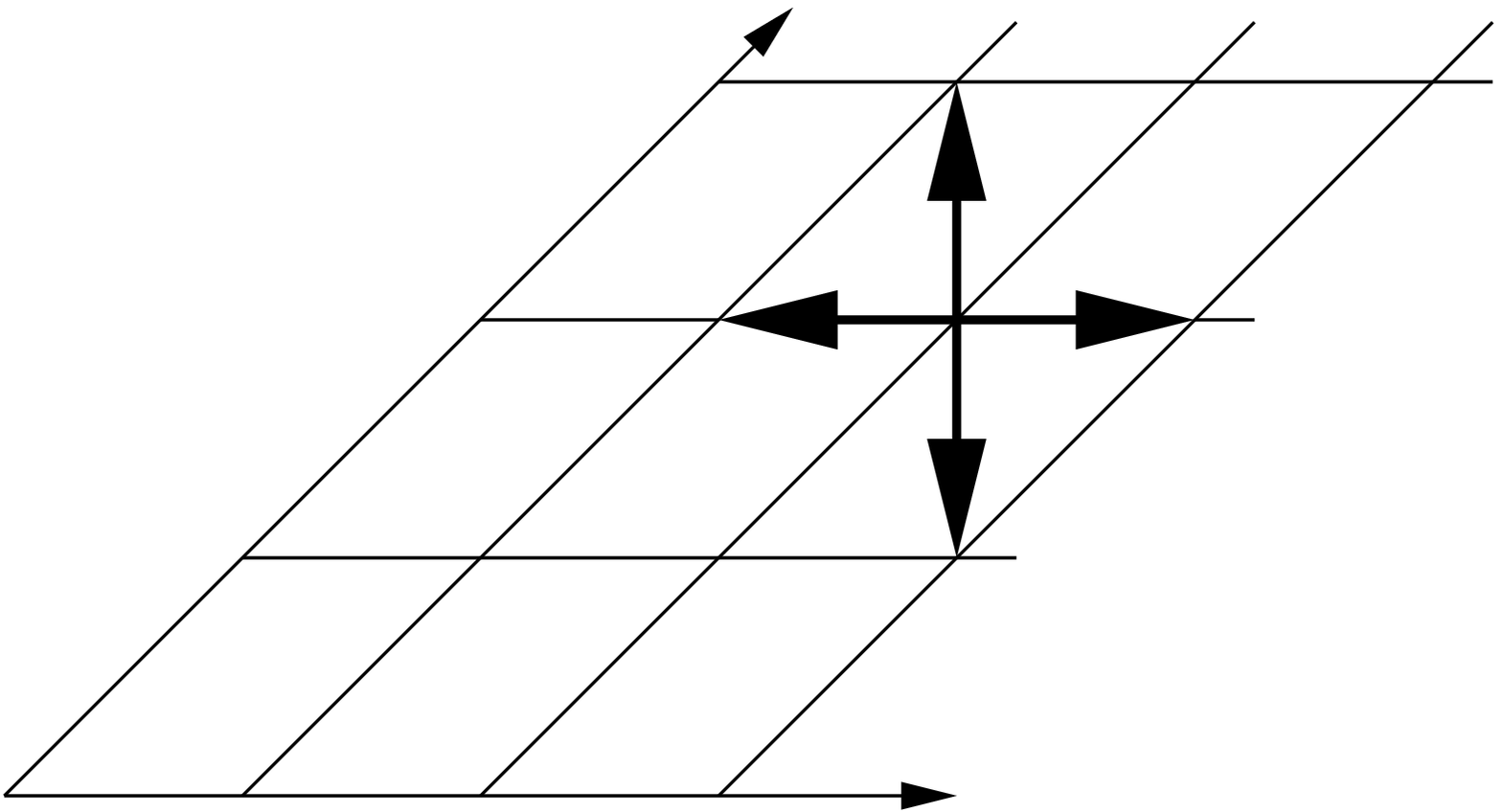}
 \end{picture}
 \caption{Random walks in the Weyl chamber of the duals of
 SU$(3)$ and Sp$(4)$}
 \label{sp4-Lie}
 \end{figure}

   By obvious transformations
   of these lattices it is immediate that
  both are killed nearest neighbors  random walks in the quarter-plane ${\mathbb Z}_+^2$
  as below.
% The former is
%   with three probabilities of jumps $1/3$ from each interior point and the latter is with
%  four probabilities $1/4$ respectively, they are pictured on Figure  .

 \begin{figure}[!ht]
 \begin{picture}(10.00,70.00)
 \hspace{20mm}
 \includegraphics{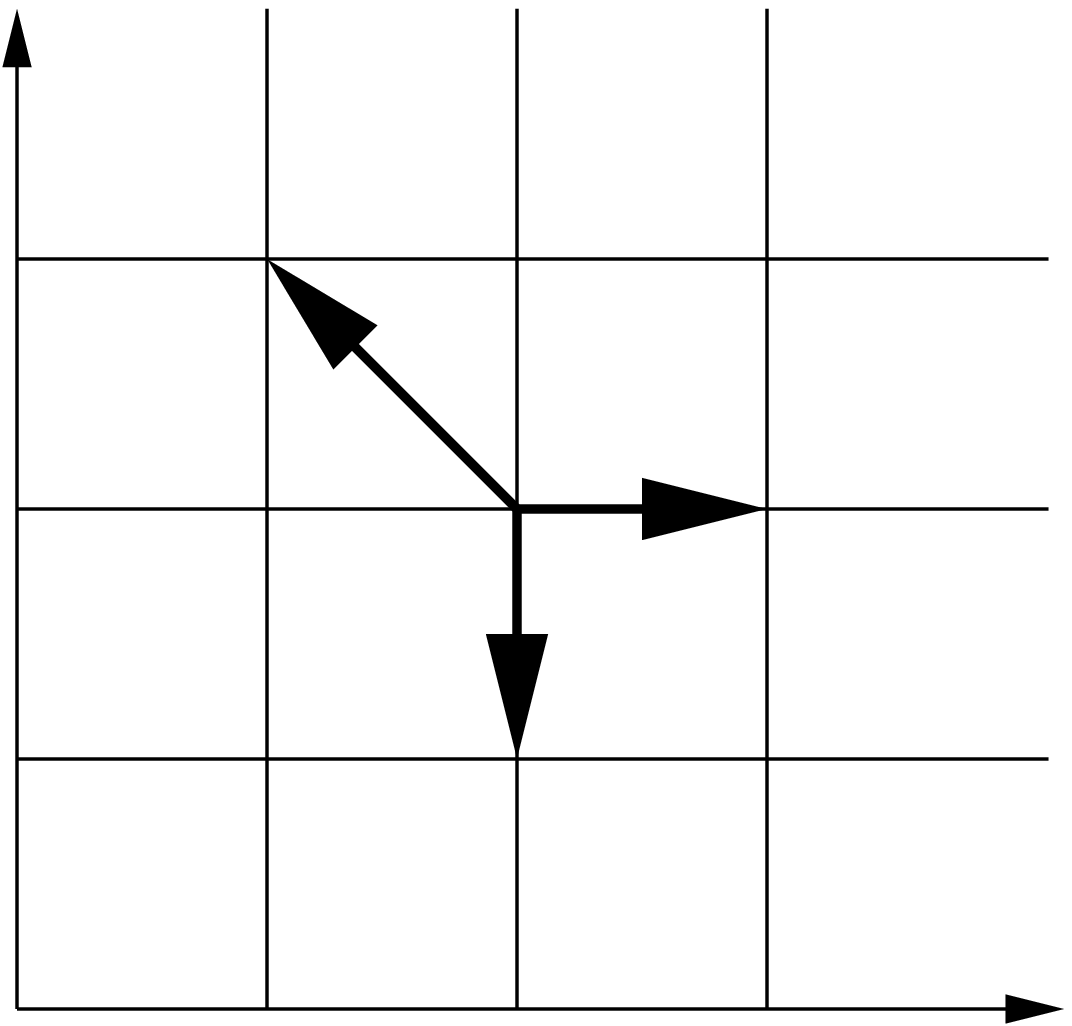}
 \hspace{70mm}
 \includegraphics{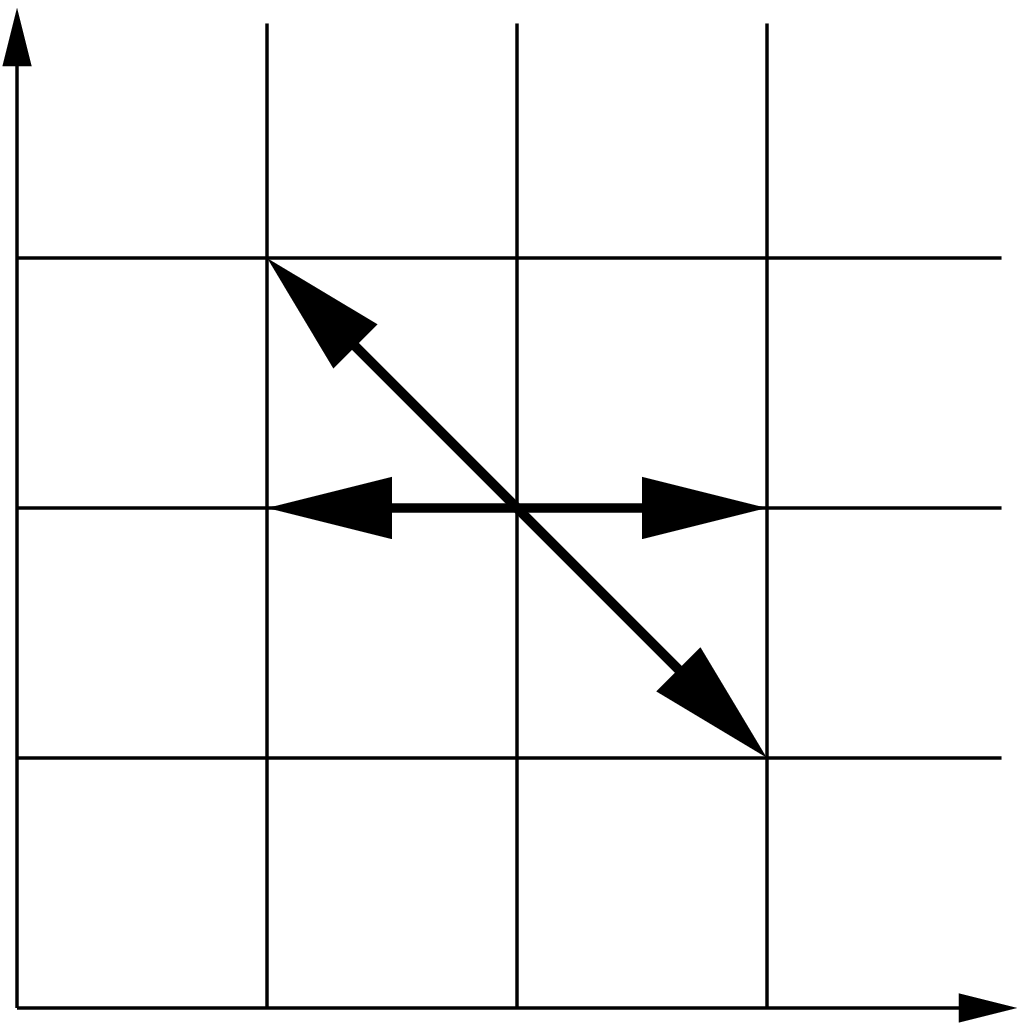}
 \end{picture}
 \caption{Random walks of Figure \ref{sp4-Lie} can be viewed as random walks in $\mathbb{Z}_{+}^{2}$}
 \label{sp4-view}
 \end{figure}

     For the random walk associated with  SU$(3)$ (and besides also
     for its multi-dimensional analogues SU$(d)$),
   Biane \cite{Bi1} computes the asymptotic of the Green functions
 along all paths of states {\it except for the ones approaching the axes}.
   To complete the latter results for these particular paths,
    the complex analysis methods  of \cite{sl3} recently turned out to be
   fruitful.
       Furthermore, in \cite{sl3} we compute the exact asymptotic
     of the Green functions along all paths and also the asymptotic of the
     absorption probabilities for a whole class of nearest neighbors random walks
     killed at the boundary of $\mathbb{Z}_{+}^2$.
       This class, including the walk associated with SU$(3)$, is characterized by the fact
         that each of its elements has a harmonic function of order three,
      i.e.\ $(i_0,j_0)\mapsto i_0 j_0(i_0+\alpha j_0+\beta)$ for some arbitrary $\alpha$ and $\beta$.
      Similar results for the simpler class of walks
      %having 
      admitting the harmonic function of order two $(i_0,j_0)\mapsto i_0 j_0$
      are derived in~\cite{mythesis}.

   The methods of \cite{sl3,mythesis} heavily rely on the analytic approach
   for the random walks in the quarter-plane developed by
   Fayolle, Iasnogorodski and Malyshev \cite{FIM}. An important notion involved
   there is that of the \emph{group of the random walk},
   which is a group of automorphisms of an algebraic
   curve; it is here properly defined in Subsection \ref{cons}.
   In this context our previous works \cite{sl3,mythesis} treat the classes of walks
   with groups of the smallest orders four and six.

   The main subject of this paper is to obtain, by this approach, the exact asymptotic of
   the Green functions along all paths as well as that of the absorption
   probabilities along the axes
   for a certain class of random walks in
   $\mathbb{Z}_+^2$ absorbed at the axes,
    with zero
   drift and containing the walk associated with Sp$(4)$---drawn on the right of Figure \ref{sp4-view}.
   This class will include walks with groups of all finite orders $2n$ for $n\geq 3$.
     We now define this class and specify the notion of group of the random walk.

%This random walk is also partially the subject of~\cite{DO}, where
%the authors are interested in making explicit the exit time queues
%associated with some processes.

 %This approach lead to explicit expression for the Green functions of killed walks,
 %and indeed this is the way we shall proceed here.

%Note that we develop in \cite{KR} this approach to the case of the generating
%functions of the Green functions for some walks
%in the quarter-plane $\mathbb{Z}_{+}^{2}$ killed at the boundary and having a positive
%drift. In \cite{sl3}, we begin to extend this approach
%to the case of killed random walks having a drift zero inside of
%the quadrant.
%We are going here to pursue its development; by the way, Subsection~\ref{sp4-Introduction}
%is strongly inspired
%by~\cite{FIM}.

\subsection{Random walks under consideration and group of the walk}
\label{cons}

Consider the random walk $(X(k),Y(k))_{k\geq 0}$ spatially
homogeneous inside of the quarter-plane $\mathbb{Z}_{+}^{2}$ and
such that if
$p_{i,j}=\mathbb{P}[(X(k+1),Y(k+1))=(i_{0}+i,j_{0}+j) \mid
(X(k),Y(k))= (i_{0},j_{0})]$, then:
%(see on the right)

\smallskip

\smallskip

\begin{enumerate}
     \item[\textnormal{(H1)}] {  $p_{1,0}+p_{1,-1}+p_{-1,0}+p_{-1,1}=1$,
                               $p_{1,0}=p_{-1,0}$, $p_{1,-1}=p_{-1,1}$;}

\smallskip

     \item[\textnormal{(H2)}] { $\{ (i,0):\ i\geq 1 \} \cup \{ (0,j):\ j\geq 1 \}$
                 is absorbing.}
%\smallskip
%
%     \item[\textnormal{($\widetilde{\text{H}3}$)}] {\it $p_{1,0}=p_{-1,0}=1/4$ and $p_{1,-1}=p_{-1,1}=1/4$
%     \footnote{We call this assumption ($\widetilde{\text{H}3}$) and not (H3) because
%     our main interest in this article will be in family of processes satisfying the hypotheses (H1), (H2) and (H3), where
%     (H3) is more general than ($\widetilde{\text{H}3}$).}.}
\end{enumerate}

\vspace{-28.5mm}

 \begin{figure}[!ht]
 \begin{picture}(10.00,75.00)
 \hspace{125mm}
 \includegraphics{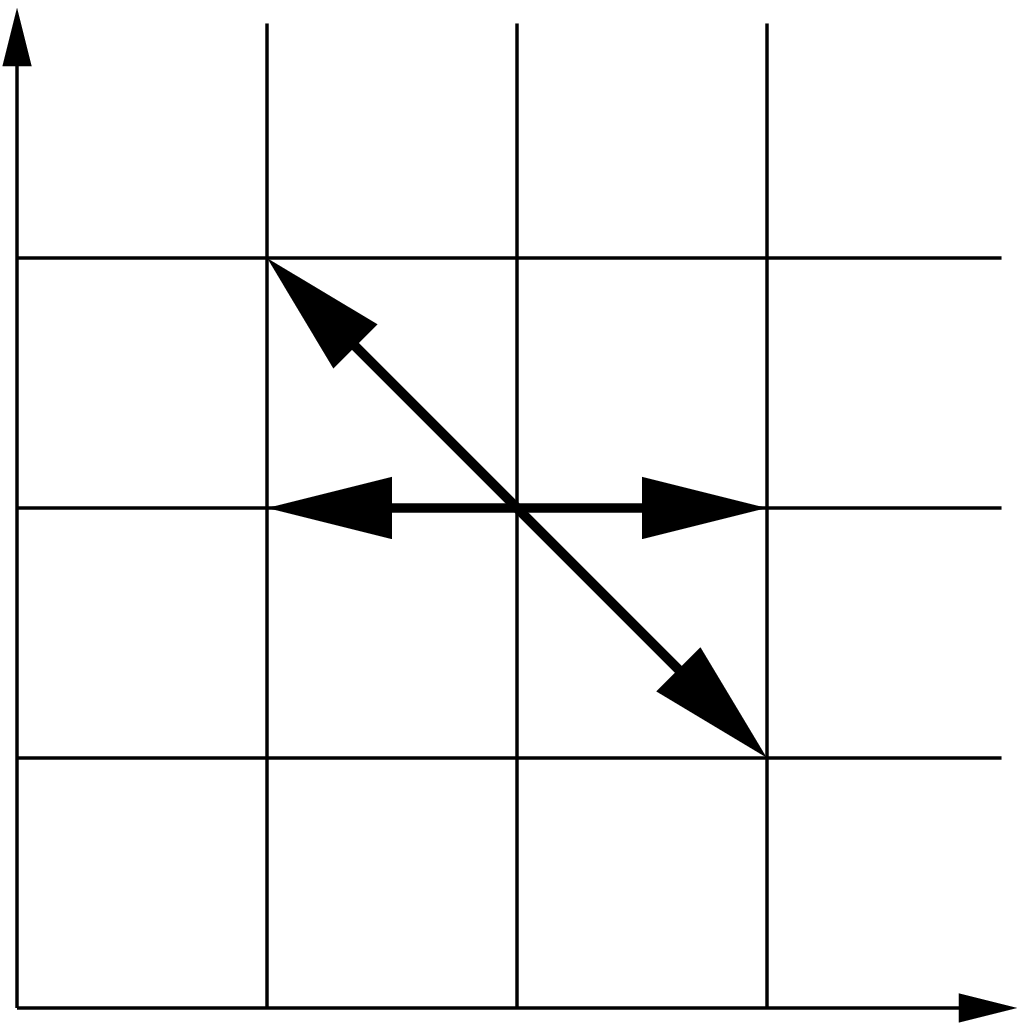}
 \end{picture}
 \end{figure}

\vspace{-2mm}

%The authors of \cite{FIM} highlight the importance and the notion of
%the \emph{group of the random walk}. Let us introduce this group in
%details in the case of the walks satisfying the hypothesis (H1).
{\raggedright Let us also define the polynomial $Q$ (which  is just a simple
transformation of the transition probabilities generating function)
by:}
     \begin{equation}
     \label{sp4-def_Q}
          Q(x,y)= x y\big[ p_{1,0}x+p_{-1,0}/x+
          p_{1,-1}x/y+p_{-1,1}y/x  -1 \big].
     \end{equation}
     If $Q(x,y)=0$, then with~(\ref{sp4-def_Q}) it is immediate that 
$Q(\widehat{\xi}(x,y))=0$ and $Q(\widehat{\eta}(x,y))=0$, where
     \begin{equation*}
     \widehat{\xi}(x,y)=\left(x,\frac{x^{2}}{y}\right),\ \ \ \ \ 
     \widehat{\eta}(x,y)=\left(\frac{p_{-1,1}y+p_{-1,0}}{p_{1,0}y+p_{1,-1}}
     \frac{y}{x},y\right).
     \end{equation*}
     The group of the walk is then
$W=\langle\widehat{\xi},\widehat{\eta}\rangle$, the group of
automorphisms of the algebraic curve $\{(x,y)\in (\mathbb{C}\cup
\{\infty\})^{2}: Q(x,y)=0\}$
 generated by $\widehat{\xi}$ and $\widehat{\eta}$. Its order is always
 even and larger than or equal to four.
It is already known \cite{BMM} that if
$p_{1,0}=p_{-1,0}=1/4$~and~$p_{1,-1}=p_{-1,1}=1/4$ then $W$ has order
eight.

More generally, we prove in Remark \ref{sp4-assumption_parameters_necessary}
that the group $W$ is finite if and only
if there exists some rational number $r$ such that $p_{1,0}=p_{-1,0}=\sin(r\pi)^{2}/2$
and $p_{1,-1}=p_{-1,1}=\cos(r\pi)^{2}/2$.

As illustrated by the works already mentioned \cite{FIM,BMM,sl3,mythesis}
the notion of finite group is nowadays extensively studied, notably because 
it often leads to worthwhile results. Let ${\mathscr P}_{2n}$ be the class 
of random walks satisfying (H1), (H2) and with a group of order $2n$. We 
show in Subsection~\ref{sp4-Galois_automorphisms} that the random walk under 
hypotheses (H1), (H2) and (H3), where
\begin{description}
\item[\textnormal{(H3)}]
$p_{1,0}=p_{-1,0}=\sin(\pi/n)^{2}/2$ {and}
$p_{1,-1}=p_{-1,1}=\cos(\pi/n)^{2}/2$,
\end{description}
belongs to ${\mathscr P}_{2n}$.

  {\it In this article we study the class being made up of the union for $n\geq 3$ of the random walks
satisfying (H1), (H2) and (H3).
  This class contains one---and only one---representative of ${\mathscr
  P}_{2n}$ for any $n\geq 3$. Precisely, for all walks in this class we
  compute the exact asymptotic of the Green functions along all paths
  and that of the absorption probabilities along the axes. This is the first
   result of that kind for random walks with zero drift
    and groups of all finite orders, up to our knowledge.}

  In the particular case $n=3$,  the process coincides with that represented
   on the left of Figure~2 in \cite{sl3} for $\alpha=2$.
  In the case $n=4$, this is the walk in the Weyl chamber of Sp$(4)$
   with jump probabilities $1/4$ studied by Biane \cite{Bi3}, see Figures \ref{sp4-Lie} and \ref{sp4-view}.

The hypothesis that $n$ is integer is technical.
Indeed, independently of this assumption, the approach \cite{FIM}---we shall use here---always yields explicit
expressions for the Green functions,
notably in terms of solutions to boundary value problems of Riemann-Hilbert type.
In the general case, these
formulations are so complex that we are not able to obtain their
asymptotic. However, they may admit a nice simplification as a closed
expression (and then in terms of the orbit under the group of a simple
function); this actually happens if the walk admits a finite
group, and if in addition some technical assumption
holds---related to fundamental domains, see Subsection \ref{sp4-Galois_automorphisms}. As it
will be properly showed in
Remark~\ref{sp4-assumption_parameters_necessary}, in the case of the
walks satisfying to (H1) this exactly implies (H3).
%In addition, the approach used in this paper can be extended to any random walk
%with jumps to the eight nearest neighbors, killed at the boundary of $\mathbb{Z}_{+}^{2}$ and
%satisfying the two same hypotheses (finite group and fundamental
%domains conditions).

%\medskip

%The family constituted by the union, for $n\geq 3$, of the walks satisfying
%(H1) and (H3) is therefore in some sense representative of all the
%random walks having a finite group~$W$. Indeed, among the set of all
%the walks with an underlying group $W$ of finite order, there is in
%this family a representative for the subclass of the walks having a
%group $W$ of order~$2n$, for any $n\geq 3$~--~let us recall that the
%case of the processes such that the group $W$ has~order $4$
%corresponds essentially to the product case  \cite{Ras}. This
%representativeness property seemed to us quite important and for
%this reason, we have thought that it was interesting not only to
%deal with the particular case $n=4$, i.e.\ with the walk in the Weyl
%chamber of the dual of Sp$(4)$, but rather with the case of~all
%these walks, for $n\geq 3$.

\subsection{Main results}
Here and throughout, $(X,Y)=(X(k),Y(k))_{k \geq 0}$
denotes the process defined by (H1), (H2) and (H3). Our first result
deals with the asymptotic of the Green functions, properly
defined by
     \begin{equation}
     \label{sp4-definition_Green_functions}
          G^{i_{0},j_{0}}_{i,j} = \mathbb{E}_{\left(i_{0} , j_{0}\right)}
          \Bigg[ \sum_{k\geq 0}\textbf{1}_{\left\{
          \left(X\left(k\right),Y\left(k\right)\right)=
          \left(i,j \right) \right\}}\Bigg].
     \end{equation}
%where $\mathbb{E}_{(i_{0},j_{0})}$ denotes the
%conditional expectation given $(X(0),Y(0))=(i_{0},j_{0})$.
Let $f_{n}$ be the function defined in~(\ref{sp4-def_f_n}); in
Section~\ref{sp4-Harmonic_functions} we shall write it explicitly
and we shall prove that it is harmonic  for
$(X,Y)$, positive inside of $\mathbb{Z}_{+}^{2}$
and equal to zero on the boundary.
\begin{thm}
\label{sp4-Main_theorem_Green_functions}
The Green functions~(\ref{sp4-definition_Green_functions}) admit the following
asymptotic as $i+j\to \infty$ and $j/i\to \tan(\gamma)$,
$\gamma\in[0,\pi/2]$:
     \begin{equation}
     \label{sp4-main_theorem_Green_functions}
          G_{i,j}^{i_{0},j_{0}} \sim
          \frac{2}{\pi} \frac{\left(n-1\right)!}
          {4^{n}\sin\left(2\pi/n\right)}
          f_{n}\left(i_{0},j_{0}\right)
          \frac{
          \sin\big(n\arctan\big[\frac{j/i}{1+j/i}
          \tan(\pi/n)\big]\big)
          }
          {
          \big[\cos(\pi/n)^{2}\left(i^{2}
          +2 i j\right)+j^{2}\big]^{n/2}
          }.
     \end{equation}
\end{thm}

\begin{rem}
\label{sp4-rem_Nn} Let
$N_{n}(j/i)=\sin\big(n\arctan\big[\frac{j/i}{1+j/i}\tan(\pi/n)\big]\big)$
be the quantity appearing in the
asymptotic~(\ref{sp4-main_theorem_Green_functions}). Let also $\gamma$ be
in $[0,\pi/2]$ and suppose that $j/i$ goes to $\tan(\gamma)$.

If $\gamma\in]0,\pi/2[$, then $N_{n}(j/i)$ goes to
$N_{n}(\tan(\gamma))$, which belongs to $]0,\infty[$.

If $\gamma=0$ or $\gamma=\pi/2$, then
$N_{n}(j/i)$ goes to $0$. More precisely,
$N_{n}(j/i)= n\tan(\pi/n)[j/i+ O(j/i)^{2}]$
if $\gamma=0$ and $N_{n}(j/i)=(n\sin(2\pi/n)/2)
[i/j+ O(i/j)^{2}]$ if $\gamma=\pi/2$.
%Note that $N_{n}$ is maximal, i.e.\ $n\arctan((j/i)/(1+j/i)\tan(\pi/n)=\pi/2$
%for $j/i=\cos(\pi/n)$ (this is just a little calculation).
\end{rem}
Theorem \ref{sp4-Main_theorem_Green_functions} has the following immediate \cite{Dynkin}
consequence.
\begin{cor}
\label{sp4-corollary_Martin_boundary}
The Martin compactification is the one-point compactification.
%In other words, there exists, up to a multiplicative constant,
%only one function which is harmonic, positive inside of the quadrant and
%equal to zero at the boundary.
\end{cor}

This paper therefore gives  a partial answer, for $d=2$, to the open
problem highlighted by  Ignatiouk-Robert in~\cite{III}, since
Corollary~\ref{sp4-corollary_Martin_boundary} implies that
up to the positive multiplicative constants, there is only one
positive harmonic function for $(X,Y)$.

Theorem \ref{sp4-Main_theorem_Green_functions} also has
a consequence on the absorption probabilities
     $
          \mathbb{P}_{(i_{0} , j_{0})}
          [(X,Y) $ is killed at $ (i,0)]$ and
          $\mathbb{P}_{(i_{0} , j_{0})}
          [(X,Y) $ is killed at $ (0,j)]$.
Indeed, by using the obvious equalities
     \begin{align*}
          \mathbb{P}_{(i_{0},j_{0})}[(X,Y) \text{ is killed at } \hspace{0.5mm}(i,0)]&=p_{1,-1}G_{i-1,1}^{i_{0},j_{0}},\\
          \mathbb{P}_{(i_{0},j_{0})}[(X,Y) \text{ is killed at } (0,j)]&=
          p_{-1,1}G_{1,j-1}^{i_{0},j_{0}}+p_{-1,0}G_{1,j}^{i_{0},j_{0}}
     \end{align*}
as well as Theorem \ref{sp4-Main_theorem_Green_functions} and
Remark \ref{sp4-rem_Nn}, we come to the following result.

\begin{cor}
\label{sp4-Main_theorem_absorption_probabilities}
The absorption probabilities
admit the following asymptotic as $i,j\to \infty$:
     \begin{align*}
     \label{sp4-main_theorem_absorption_probabilities}
          \mathbb{P}_{(i_{0},j_{0})}[(X,Y) \text{ is killed at }  \hspace{0.5mm}(i,0)]
          &\sim \frac{1}{2\pi} \frac{n!}
          {\left[4\cos\left(\pi/n\right)\right]^{n}}
          f_{n}\left(i_{0},j_{0}\right)
          \frac{1}{i^{n+1}},\\
          \mathbb{P}_{(i_{0},j_{0})}[(X,Y) \text{ is killed at }(0,j)]
          &\sim \frac{1}{2\pi} \frac{n!}
          {4^{n}}
          f_{n}\left(i_{0},j_{0}\right)
          \frac{1}{j^{n+1}}.
     \end{align*}
\end{cor}

% -- recall that the family~of
%processes considered here is representative of the set of all
%two-dimensional killed random walks having a finite group.

%Extending the results of this paper to all  random walks
%(i.e.\ having eventually
%an infinite group)
%to the eight nearest neighbors, with drift zero and killed
%at the boundary of $\mathbb{Z}_{+}^{2}$ seems to us, at this time, a
%quite difficult challenge.

\subsection{Harmonic functions and link with Brownian motion}
\label{sub_harmo_bm}

Let us now have a closer look at the harmonic function
$f_{n}$ that governs the asymptotic (\ref{sp4-main_theorem_Green_functions})
of the Green functions (\ref{sp4-Main_theorem_Green_functions}).
%We shall first present some of its properties and then we will study
%its links with harmonic functions of Brownian motion.
All the results of Subsection \ref{sub_harmo_bm} are proven in
Section \ref{sp4-Harmonic_functions}.

\begin{prop}
\label{new_prop_properties_fn}
$\phantom{cc}$
     \begin{enumerate}[label=\textnormal{(\roman{*})},ref=\textnormal{(\roman{*})}]
          \item \label{sp4-it_polynomial}
          $f_{n}$ is a real polynomial in the variables $i_{0},j_{0}$ of degree
          exactly $n$;
          \item \label{sp4-it_harmo}
          $f_{n}$ is a harmonic function for the process $(X,Y)$;
          %\item \label{sp4-it_value_1_1}
          %$f_{n}(1,1)=8 n^{3}/\tan(\pi/n)$.
          \item \label{sp4-it_zero_boundary}
          $f_{n}(i_{0},0)=f_{n}(0,j_{0})=0$ for all integers $i_{0}$ and $j_{0}$;
          \item \label{sp4-it_positive}
          If $i_{0},j_{0}>0$ then $f_{n}(i_{0},j_{0})>0$.
     \end{enumerate}\end{prop}
The explicit formulation of $f_{n}$ for general values of $n$---that we shall obtain in Section~\ref{sp4-Harmonic_functions}--- being quite complex,
here we just give the following three examples:
     \begin{eqnarray*}
          f_{3}(i_{0},j_{0})&=&\hspace{6mm}24\cdot 3^{1/2}\cdot i_{0} j_{0} (i_{0}+2j_{0}),\\
          f_{4}(i_{0},j_{0})&=&\hspace{6.5mm}(256/3)\cdot i_{0} j_{0} (i_{0}+2j_{0}) (i_{0}+j_{0}),\\
          %f_{5}(i,j)=\frac{10}{3}\sqrt{65+29\sqrt{5}}
          %i j (i+2j)\left((i+j( \sqrt{5}-1))(i+j(3-\sqrt{5}))+
          %(7\sqrt{5}-15)/2\right)
          f_{6}(i_{0},j_{0})&=&(288/5) 3^{1/2}\cdot  i_{0} j_{0}(i_{0}+2j_{0})(i_{0}+j_{0})
          \big((i_{0}+2j_{0}/3)(i_{0}+4j_{0}/3)+10/9\big).
     \end{eqnarray*}
Note that for $n=3$ and $n=4$, $f_n$ is a homogeneous function of degree $n$
(i.e.~$f_{n}(\lambda x,\lambda y)=\lambda^{n} f_{n}(x,y)$), while
$f_6$ is not. In fact the next result holds.
\begin{prop}
\label{prop_non_homo} For any $n\geq 5$, $f_n$ is not homogeneous.
\end{prop}
Let us conclude the introduction by outlining the link of $f_n$
with the harmonic functions of Brownian motion. Let
     \begin{equation}
     \label{def_trans}
          \phi(x,y)=\big((x+y)/\sin(\pi/n),
          y/\cos(\pi/n)\big).
     \end{equation}
Then the random walk $\phi(X,Y)$ has an identity
covariance and takes its values in the cone $\Lambda(0,\pi/n)=\{t
\exp(\imath \theta): 0\leq t\leq
          \infty, 0\leq \theta\leq \pi/n\}$.

 \begin{figure}[!ht]
 \begin{picture}(10.00,70.00)
 \hspace{20mm}
 \includegraphics{aaaXL2.eps}
 \hspace{70mm}
 \includegraphics{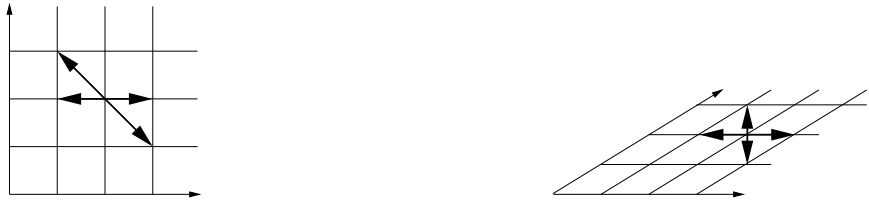}
 \end{picture}
 \caption{On the left, the walk $(X,Y)$; on the right, the walk
 $\phi(X,Y)$, with $\phi$ defined in (\ref{def_trans})}
 \label{tra}
 \end{figure}

           $\phi(X,Y)$ therefore lies in the domain of attraction of the standard Brownian motion
          killed at the boundary of $\Lambda(0,\pi/n)$.

For the Brownian motion, it is well-known that there is only one
harmonic function $h$ positive inside of a given cone and vanishing
on the boundary: it is called the r\'eduite \cite{MR513885} of the cone.
It happens to be homogeneous \cite{MR513885}. When the cone is $\Lambda(0,\pi/n)$,
the r\'eduite is equal to $h(\rho\exp(\imath \theta))
=\rho^{n}\sin(n\theta)$.

Moreover, the asymptotic of the Green functions
of the Brownian
motion killed at the boundary of $\Lambda(0,\pi/n)$ can be obtained
from \cite{Smits} and is equal to:
     \begin{equation*}
     \label{asymp_GF_MB}
          G^{\rho\exp(\imath \theta)}_{r \exp(\imath \eta)} \sim
          \frac{2}{\pi^{1/2}}h(\rho\exp(\imath \theta))\frac{\sin(n \eta)}{r^n},
          \ \ \ \ \ r\to \infty.
     \end{equation*}

\begin{prop}
\label{prop_link} Let $\phi$ be defined in~(\ref{def_trans}).
 Up to a multiplicative constant, the homogeneous function
$h(\phi(i_0,j_0))$ equals the dominant term of the non-homogeneous
harmonic function of the random walk, i.e.
\begin{equation*}
\label{fff}
 f_{n}(i_0,j_0)=h(\phi(i_0,j_0))[1+o(1)],
 \ \ \ \ \ i_0,j_0 \to\infty.
\end{equation*}
\end{prop}

Proposition \ref{prop_link} will follow from our results by a
direct computation, see Section~\ref{sp4-Harmonic_functions}.
Let us also note, as in \cite{EJP2010-37}, that this 
proposition
is in accordance with Donsker's invariant~principle.

\begin{rem}
\label{rwMB}
   Proposition \ref{prop_link} also entails
   that the comparison approach of random walks by Brownian
   motion sketched in the first part of the introduction can
   neither give the approxi- mation of the Green functions with the precision
    of Theorem \ref{sp4-Main_theorem_Green_functions} nor even
  to specify the unique harmonic function for our class of random walks  (H1), (H2) and (H3).

     Furthermore, from the asymptotic results (\ref{sp4-main_theorem_Green_functions})
     we notice that
      the Green functions can tend to zero arbitrarily
     fast whenever the order of the group is taken high enough, while
     the Beurling estimates \cite{Kestenn,LL} relating the random walk
    to the Brownian motion typically have a polynomial precision---which in addition depends only on the dimension.
\end{rem}

The rest of the paper is organized as follows. In Section \ref{sp4-h_x_z}
we find explicitly the~absorption probabilities and the Green functions
(\ref{sp4-definition_Green_functions}). In Section \ref{sp4-Harmonic_functions},
we then study precisely the harmonic function $f_n$. Finally in Section
\ref{sp4-Martin_boundary} we prove Theorem \ref{sp4-Main_theorem_Green_functions}.

\section{Expression of the absorption probabilities and of the Green functions}
\label{sp4-h_x_z}

\subsection{A functional equation between the generating functions}
\label{sp4-Introduction}
Subsection~\ref{sp4-Introduction}
consists in preparatory results and is inspired by the book~\cite{FIM}.
Define
     \begin{equation}
     \label{sp4-def_generating_functions}
              \begin{array}{rcl}
              \displaystyle G^{i_{0},j_{0}}(x,y) &=& \displaystyle \sum_{i,j\geq 1} G_{i,j}^{i_{0},j_{0}}x^{i-1} y^{j-1},\\
              \displaystyle h^{i_{0},j_{0}}(x) &=&\displaystyle
               \sum_{i\geq 1}^{ }  \mathbb{P}_{(i_{0},j_{0})}[(X,Y) \text{ is killed at } (i,0)] x^{i},\\
              \displaystyle \widetilde{h}^{i_{0},j_{0}}\left(y\right) &=& \displaystyle
               \sum_{j\geq 1}^{ } \mathbb{P}_{(i_{0},j_{0})}[(X,Y) \text{ is killed at } (0,j)] y^{j}
               \end{array}
     \end{equation}
     the generating functions of the Green functions (\ref{sp4-definition_Green_functions})
     and of the absorption probabilities.
With these notations, we can state the following functional equation:
     \begin{equation}
     \label{sp4-functional_equation}
          Q\left(x,y\right)G^{i_{0},j_{0}}\left(x,y\right) =
          h^{i_{0},j_{0}}\left(x\right)+
          \widetilde{h}^{i_{0},j_{0}}\left( y \right)
          -x^{i_{0}} y^{j_{0}},
     \end{equation}
$Q$ being defined in~(\ref{sp4-def_Q}). {A priori}, Equation
(\ref{sp4-functional_equation}) has a meaning in $\{(x,y)\in
\mathbb{C}^{2}: |x|<1, |y|<1 \}$. The proof of
(\ref{sp4-functional_equation}) is obtained exactly as in
Subsection~2.1 of \cite{KR}.

When no ambiguity on the initial state can arise, we will drop
the index $i_{0},j_{0}$ and we will write $G_{i,j},G(x,y),h(x),
\widetilde{h}(y)$ for $G^{i_{0},j_{0}}_{i,j},G^{i_{0},j_{0}}(x,y),
h^{i_{0},j_{0}}(x),\widetilde{h}^{i_{0},j_{0}}(y)$.
%comes from writing that for $k,l,n\in
%(\mathbb{Z}_{+})^2$,
%               \begin{eqnarray*}
%                \mathbb{P}((X(n+1),Y(n+1))=(k,l))=
%                \sum_{i,j\geq 1} \mathbb{P}\left(\left(X\left(n\right),Y\left(n\right)\right)
%                    =\left(i,j\right) \right)
%                    p_{\left(i,j\right),\left(k,l\right)}+\hspace{20mm}\\
%                    +\sum_{i\geq 1}\mathbb{P}\left(\left(X\left(n\right),Y\left(n\right)\right)
%                    =\left(i,0\right) \right)
%                    \delta_{\left(k,l\right)}^{\left(i,0\right)}+
%                    \sum_{j\geq 1}\mathbb{P}\left(\left(X\left(n\right),Y\left(n\right)\right)
%                    =\left(0,j\right) \right)
%                    \delta_{\left(k,l\right)}^{\left(0,j\right)} +    \\
%                    +\mathbb{P}\left(\left(X\left(n\right),Y\left(n\right)\right)
%                    =\left(0,0\right) \right)
%                    \delta_{\left(k,l\right)}^{\left(0,0\right)},\hspace{20mm}\hspace{20mm}\hspace{5mm}
%               \end{eqnarray*}
%where $\delta^{(i,j)}_{(k,l)}=1$ if $i=k$ and
%$j=l$, otherwise $0$. It remains to multiply by $x^{k} y^{l}$ and
%then to sum with respect to $k,l,n$.

\medskip

Let us now have a look to the algebraic curve
$\{(x,y)\in (\mathbb{C}\cup \{\infty\})^{2}: Q(x,y)=0\}$, that we note
$\mathscr{Q}$ for the sake of briefness.
Start by writing the polynomial~(\ref{sp4-def_Q}) alternatively
     \begin{equation}
     \label{sp4-def_H_alternative}
          Q\left(x,y\right) = a\left(x\right) y^{2}+ b\left(x\right) y
          + c\left(x\right) = \widetilde{a}\left(y\right) x^{2}+
          \widetilde{b}\left(y\right) x + \widetilde{c}\left(y\right),
     \end{equation}
where
$a(x) = p_{1,-1}$, $b(x) = p_{1,0}x^{2}-x+p_{1,0}$,
$c(x)= p_{1,-1}x^{2}$ and $\widetilde{a}(y) = p_{1,0}y +p_{1,-1}$,
$\widetilde{b}(y) =-y$, $\widetilde{c}(y) = p_{1,-1} y^{2}+p_{1,0}y$.
%     \begin{equation*}
%          \left.\begin{array}{ccccccccccccccccc}
%               a\left(x\right) &=& & &p_{1,-1},
%               &&b\left(x\right) &=& p_{1,0}x^{2}-x+p_{1,0},
%               &&c\left(x\right) &=& p_{1,-1}x^{2}\phantom{p_{iiiiiiiii}},\\
%               \widetilde{a}\left(y\right) &=& p_{1,0}y &+&p_{1,-1},
%               &&\widetilde{b}\left(y\right) &=& \phantom{p_{iii}x^{2}}-y\phantom{p_{iiiiiiii}},
%               &&\widetilde{c}\left(y\right) &=& p_{1,-1} y^{2}+p_{1,0}y.
%          \end{array}\right.
%     \end{equation*}
Set also $d(x)=b(x)^{2}-4a(x)c(x)$ and
$\widetilde{d}(y)=\widetilde{b}(y)^{2}-4\widetilde{a}(y)\widetilde{c}(y)$. We have
     \begin{equation}
     \label{sp4-def_d_d_tilde}
          d\left(x\right) = p_{1,0}^{2}\big(x-1\big)^{2}
          \big(x^{2}+2x(1-1/p_{1,0})+1\big),\ \ \ \ \
          \widetilde{d}(y)=-4p_{1,0}p_{1,-1}y\big(y-1\big)^{2}.
     \end{equation}

The polynomial $d$ has manifestly a double root at $1$ and two simple roots at
positive points, that we denote by $x_{1}<1<x_{4}$.
As for $\widetilde{d}$, it has a double root at $1$ and a simple
root at $0$. We also note $y_{1}=0$ and $y_{4}=\infty$.

Then with~(\ref{sp4-def_H_alternative}) we notice that
$Q(x,y)=0$ is equivalent to $[b(x)+2a(x)y]^{2}=d(x)$
or to $[\widetilde{b}(y)+2\widetilde{a}(y)x]^{2}=\widetilde{d}(y)$.
It follows from the particular form of
$d$ or $\widetilde{d}$, see (\ref{sp4-def_d_d_tilde}), that
the surface $\mathscr{Q}$ has genus zero and is thus
homeomorphic to a sphere \cite{JS}. As a consequence
this Riemann surface can be rationally uniformized, in
the sense that
it is possible to find two rational functions,  say
$\pi$ and $\widetilde{\pi}$, such that~$\mathscr{Q}
=\{(\pi(s),\widetilde{\pi}(s)): $ $s\in \mathbb{C}\cup \{\infty\} \}$.
Furthermore, as shown in Chapter~6 of~\cite{FIM}, we can
take~$\pi(s)=[x_{4}+x_{1}]/2+([x_{4}-x_{1}]/4) (s+1/s)$, $x_{1}$ and
$x_{4}$ being defined below~(\ref{sp4-def_d_d_tilde}); it is
then possible to deduce a correct expression for $\widetilde{\pi}$,
since by construction the equality $Q(\pi,\widetilde{\pi})=0$ has to hold.
For more details about the
construction of Riemann surfaces, see for instance~\cite{JS}.

\subsection{Uniformization and meromorphic continuation}
\label{sp4-Galois_automorphisms}

%Thanks to~(\ref{sp4-def_H_alternative}), for any $s=(x,y) \in T$, $\xi$
%and $\eta$ take the following explicit expressions:
%     \begin{equation}\label{sp4-def_xi_eta_Q}
%          \xi\left(x,y\right)=\left(x,\frac{c\left(x\right)}
%          {a\left(x\right)}\frac{1}{y}\right),\hspace{5mm}
%          \eta\left(x,y\right)=\left(\frac{\widetilde{c}\left(y\right)}
%          {\widetilde{a}\left(y\right)}\frac{1}{x},y\right).
%     \end{equation}
%$\xi$ and $\eta$ are of order two: $\xi^{2}=1$,
%$\eta^{2}=1$. In~\cite{Ma2,FIM}, for reasons
%explained there, they are also called Galois automorphisms.

But rather than the uniformization $(\pi,\widetilde{\pi})$ proposed in~\cite{FIM}
and recalled at the end of the previous subsection, we prefer using another,
that will turn out to be quite more convenient. This new uniformization, that
we call $(x,y)$, is just equal to $(\pi \circ L,\widetilde{\pi} \circ L)$,~where
     \begin{equation*}
          L\left(z\right) = \frac{z_{0}z - 1}{z-z_{0}},
          %z_{0} = -(1-2p_{1,0})^{1/2}+\imath (2p_{1,0})^{1/2}.
          \ \ \ \ \ z_{0}=-\exp\left(-\imath \pi/n\right).
     \end{equation*}

We notice that $z_{0}$ is such that $\pi(z_{0})=\pi(\overline{z_{0}})=
\widetilde{\pi}(z_{0})=\widetilde{\pi}(\overline{z_{0}})=1$
and that its explicit expression above is due to~(H3),
for more details see Remark~\ref{sp4-assumption_parameters_necessary}.
Then, starting from the formulations of $(\pi,\widetilde{\pi})$
and of $L$, we easily show that the expression of the new uniformization can be
%     \begin{equation}
%     \label{sp4-uniformization}
%          \left\{
%          \begin{array}{ccc}
%          x(z) &=& \displaystyle
%          \frac{(z+z_{0})(z+\overline{z_{0}})}
%          {(z-z_{0})(z-\overline{z_{0}})},\\
%          y(z) &=& \displaystyle
%          \frac{(z+z_{0})^{2}}
%          {(z-z_{0})^2},
%          \end{array}
%          \right.
%     \end{equation}
     \begin{equation}
     \label{sp4-uniformization}
          x\left(z\right) = \frac{\left(z+z_{0}\right)
          \left(z+\overline{z_{0}}\right)}
          {\left(z-z_{0}\right)\left(z-\overline{z_{0}}\right)},\ \ \ \ \
          y\left(z\right) = \frac{(z+z_{0})^{2}}
          {(z-z_{0})^{2}}.
     \end{equation}

Compared to $(\pi,\widetilde{\pi})$, this uniformization $(x,y)$ has the significant
advantage of transfor-ming the important cycles (i.e.\ the branch cuts
$[x_{1},x_{4}]$ and $[y_{1},y_{4}]$, the unit circles~$\{|x|=1\}$ and $\{|y|=1\}$)
into very simple cycles,
% --~what was not the case of $(\pi,\widetilde{\pi})$~--,
since the following equalities hold, see also Figure \ref{sp4-transformation_cycles}:
     \begin{equation}
     \label{sp4-eH_cycles}
     \begin{array}{cccccc}
          x^{-1}([x_{1},x_{4}])&=&\phantom{z_{0}}\mathbb{R}\cup\{\infty\},\ \ &
          x^{-1}(\{|x|=1\})&=&\phantom{z_{0}}\imath \mathbb{R}\cup\{\infty\},\\
          y^{-1}([y_{1},y_{4}])&=&z_{0} \mathbb{R}\cup\{\infty\},\ \ &
          y^{-1}(\{|y|=1\})&=&z_{0} \imath \mathbb{R}\cup\{\infty\}.
     \end{array}
     \end{equation}

 \begin{figure}[!ht]
 \begin{center}
 \begin{picture}(420.00,150.00)
 \includegraphics{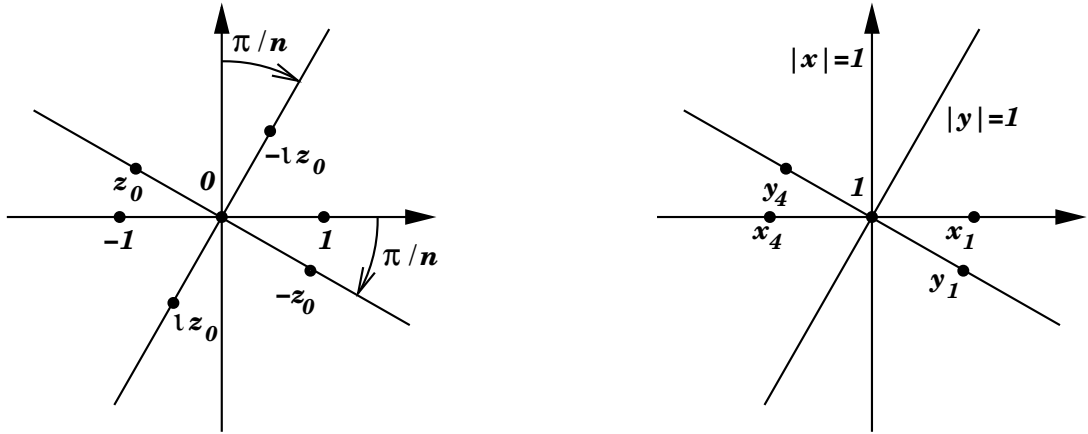}
 \end{picture}
 \end{center}
 \vspace{-3mm}
 \caption{The uniformization space $\mathbb{C}\cup\{\infty\}$, with on the left
          some important elements of it and on the right the corresponding
          elements through the coordinates $x$ and $y$}
 \label{sp4-transformation_cycles}
 \end{figure}

To obtain~(\ref{sp4-eH_cycles}), it is sufficient to use the explicit expressions of the
branch points~$x_{1},x_{4}$, $y_{1},y_{4}$, see below~(\ref{sp4-def_d_d_tilde}), as well
as the explicit formulation of the uniformization, see~(\ref{sp4-uniformization}).

\medskip

Let us go back to $\widehat{\xi}$ and $\widehat{\eta}$, the automorphisms of the algebraic
curve $\mathscr{Q}$ introduced in Subsection~\ref{cons}. Thanks to the uniformization
(\ref{sp4-uniformization}), they define two automorphisms $\xi$ and $\eta$ on $\mathbb{C}
\cup\{\infty\}$, which are determined by
     \begin{equation}
     \label{sp4-characterization_automorphisms}
          \xi^{2} = 1,\ \ \ \
          x\circ \xi = x,\ \ \ \
          y\circ \xi=\frac{x^{2}}{y},\ \ \ \
          \eta^{2} = 1,\ \ \ \
          y\circ \eta = y,\ \ \ \
          x\circ \eta =\frac{p_{1,-1}y^{2}+p_{1,0}y}{p_{1,0}y+p_{1,-1}}\frac{1}{x}.
     \end{equation}
Using the well-known characterization of the automorphisms of the
Riemann sphere~$\mathbb{C}\cup\{\infty\}$,~(\ref{sp4-uniformization})
and~(\ref{sp4-characterization_automorphisms}),
we obtain that $\xi$ and $\eta$ have the following expressions:
     \begin{equation}
     \label{sp4-def_automorphisms_C_s}
          \xi(z)  = 1/z,\ \ \ \ \
          \eta(z) = \exp(-2\imath \pi/n)/z.
     \end{equation}

The expression above of $\eta$ in terms of $n$ is due to the
assumption~(H3), see Remark~\ref{sp4-assumption_parameters_necessary}
for more details. Note also that leading to these particularly nice
analytic expressions of $\xi$ and $\eta$ is another very pleasant
property of the uniformization $(x,y)$.

As in Subsection~\ref{cons}, we call
the group generated by $\xi$ and $\eta$
     \begin{equation*}
          W_{n}=\left\langle \xi,\eta\right\rangle
     \end{equation*}
the \emph{group of the random walk}. In the context of
this article, $W_{n}$ is isomorphic to
the dihe- dral group of order $2n$, i.e.\ to the
group of symmetries of a regular polygon with $n$ sides,
$\xi$ and $\eta$ playing the role of the two reflections.
%In the sequel, when no confusion can occur,
%we will write $\eta$ instead of $\eta_{n}$.

%Since $\xi$ and $\eta$ are inversions, it is natural
%to consider their cycles of inversion. An easy calculation shows that they
%equal respectively to $\mathcal{C}(0,1)$, the unit circle, and to
%$-(1-2p_{1,0})^{1/2}+\imath \mathbb{R}$,
%the vertical line going through $-(1-2p_{1,0})^{1/2}$.

%It is interesting to note that these
%cycles of inversion correspond in fact to the branch cuts $[x_{1},x_{4}]$
%and $[y_{1},y_{4}]$, in the sense that
%$x^{-1}([x_{1},x_{4}]) = \mathcal{C}(0,1)$ and
%$y^{-1}([y_{1},y_{4}]) = -(1-2p_{1,0})^{1/2}+\imath \mathbb{R}$.

%In addition, we will soon need to know how the uniformization transforms the unit circle.
%The answer is that
%$x^{-1}(\mathcal{C}(0,1))=\mathcal{C}(-(1-2p_{1,0})^{1/2},(2p_{1,0})^{1/2})$ and
%$y^{-1}(\mathcal{C}(0,1))=\mathcal{C}(-1/(1-2p_{1,0})^{1/2},(2p_{1,0})^{1/2}/(1-2p_{1,0})^{1/2})$.

\medskip

We are now going to state and prove Proposition~\ref{sp4-continuation_h_h_tilde_covering},
which actually is the main result of Subsection~\ref{sp4-Galois_automorphisms} and that
deals with the continuation of the generating functions $h$ and $\widetilde{h}$ defined
in~(\ref{sp4-def_generating_functions}). For this, we need to describe the action
of the elements of $W_{n}$ on some cones of the plane as well as to find
some fundamental domains of the plane for the action of $W_{n}$---we say that $D$ is a fundamental domain of the plane for
the action of $W_{n}$ if $\cup_{w\in W_{n}}w(D)=\mathbb{C}$
and if in addition the latter union is disjoint.

Let us take the following notation: for $\theta_{1}\leq \theta_{2}$, let
     \begin{equation*}
          \Lambda (\theta_{1},\theta_{2}) = \big\{t \exp(\imath \theta): 0\leq t\leq
          \infty, \theta_{1}\leq \theta\leq \theta_{2}\big\}
     \end{equation*}
be the cone with vertex at $0$
and opening angles $\theta_{1},\theta_{2}$. In particular,
$\Lambda(\theta,\theta)=\exp(\imath \theta)\mathbb{R}_{+}\cup\{\infty\}$.
% is the half-line starting at $0$ and making an angle
%$\theta$ with the real axis.
Thanks to~(\ref{sp4-def_automorphisms_C_s}),
we obtain that the action of $\xi$ and $\eta$ on these cones
is simply given by
$\xi(\Lambda(\theta_{1},\theta_{2}))=\Lambda(-\theta_{2},-\theta_{1})$ and
$\eta(\Lambda(\theta_{1},\theta_{2}))=\Lambda(-\theta_{2}-2\pi/n,-\theta_{1}-2\pi/n)$.
These facts are illustrated on the left of Figure~\ref{sp4-Lambda_D}.

Define now, for $k\in \{0,\ldots ,n\}$,
          \begin{equation*}
               D_{k}^{+}=\Lambda\left(\frac{k-1}{n}\pi,\frac{k}{n}\pi\right),
               \ \ \ \ \ D_{k}^{-}=\Lambda\left(-\frac{k+1}{n}\pi,-\frac{k}{n}\pi\right).
          \end{equation*}
Sometimes, we will write $D_{0}$ instead of $D_{0}^{+}=D_{0}^{-}$ and
$D_{n}$ instead of $D_{n}^{+}=D_{n}^{-}$. Clearly,
     \begin{equation}
     \label{sp4-fundamental_domain_first_step}
          D_{0} \cup D_{n} \cup \bigcup_{k=1}^{n-1}D_{k}^{+}
          \cup \bigcup_{k=1}^{n-1}D_{k}^{-}=\mathbb{C}\cup\{\infty\}.
     \end{equation}
The definitions of $D_{k}^{+}$ and $D_{k}^{-}$, as well
as~(\ref{sp4-fundamental_domain_first_step}), are illustrated
on the right of Figure~\ref{sp4-Lambda_D}.

 \begin{figure}[!ht]
 \begin{center}
 \begin{picture}(425.00,165.00)
 \includegraphics{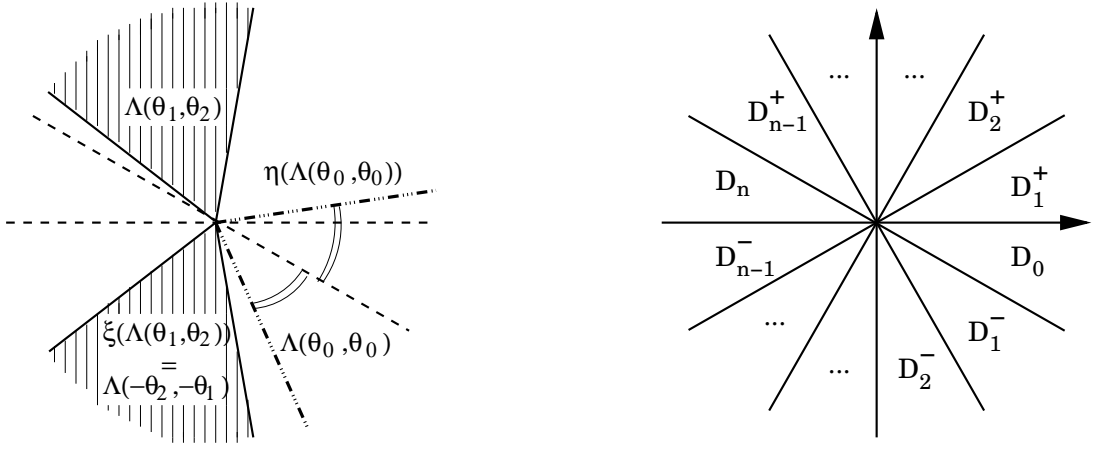}
 \end{picture}
 \end{center}
 \vspace{-3mm}
 \caption{Important cones of the uniformization space}
 \label{sp4-Lambda_D}
 \end{figure}

It is immediate that for any
$k\in\{1,\ldots ,n\}$, we have $D_{k}^{+}=\xi(D_{k-1}^{-})$ and
$D_{k}^{-}=\eta(D_{k-1}^{+})$.
In particular, for any
$2k\in\{1,\ldots ,n\}$,
     \begin{equation*}
          D_{2k}^{+}=\big((\xi\circ\eta)^{k}\big)\big(D_{0}\big),
          \ \ \ \ \ D_{2k}^{-}=\big((\eta\circ\xi)^{k}\big)\big(D_{0}\big).
     \end{equation*}
     Likewise, for any $2k+1\in\{1,\ldots ,n\}$,
     \begin{equation*}
          D_{2k+1}^{+}=\big(\xi\circ(\eta\circ \xi)^{k}\big)\big(D_{0}\big),
          \ \ \ \ \ D_{2k}^{-}=\big(\eta\circ(\xi\circ\eta)^{k}\big)\big(D_{0}\big).
     \end{equation*}

With~(\ref{sp4-fundamental_domain_first_step}), these equalities prove that
$\cup_{w\in W_{n}}w(D_{0})=\mathbb{C}\cup\{\infty\}$, in such a way that
$D_{0}$ is a fundamental domain for the action of $W_{n}$ on $\mathbb{C}$---this is not quite exact, since each half-line
$\Lambda(k\pi/n,k\pi/n)$, $k\in\{0,\ldots,2n-1\}$
appears twice in the union $\cup_{w\in W_{n}}w(D_{0})$.

\medskip

We are now able to state and prove
Proposition~\ref{sp4-continuation_h_h_tilde_covering},
after the weak recall on the lifting of functions that follows:
any function $f$ of the variable $x$ (resp.\ $y$) defined  on some
domain $D\subset \mathbb{C}$ can be lifted on $\{z
\in \mathbb{C}\cup\{\infty\}: x(z)\in D \}$ (resp.\
$\{z \in \mathbb{C}\cup\{\infty\}: y(z)\in D \}$) by
setting $F(z)=f(x(z))$ (resp.\ $F(z)=f(y(z))$).

In particular, we can lift the generating functions $h$ and
$\widetilde{h}$ considered in (\ref{sp4-def_generating_functions}) and we set $H(z)=h(x(z))$ and
$\widetilde{H}(z)=\widetilde{h}(y(z))$; they are well defined on
$\{z\in \mathbb{C}\cup\{\infty\}: |x(z)|\leq 1\}
%=\Lambda(\pi/2,3\pi/2)
$ and
$\{z \in\mathbb{C}\cup\{\infty\}: |y(z)|\leq 1\}
%=\Lambda(\pi/2-\pi/n,3\pi/2-\pi/n)
$
respectively. Consequently, on $\{z \in \mathbb{C}\cup\{\infty\}: |x(z)|\leq 1,
|y(z)|\leq 1\}$ (which equals $\Lambda(-\pi/2,\pi/2-\pi/n)$,
see Figure~\ref{sp4-transformation_cycles}), the
functional equation~(\ref{sp4-functional_equation})
yields $H(z)+\widetilde{H}(z)-x^{i_{0}}y^{j_{0}}(z)=0$---in the sequel, we will often write $x^{i_{0}}y^{j_{0}}(z)$
instead of $x(z)^{i_{0}}y(z)^{j_{0}}$.

%In~\cite{FIM}, the authors locate the cycles $|x(z)|=1$ and
%$|y(z)|=1$ on the fundamental parallelogram of
%Figure~\ref{sp4-Locations_of_the_cuts}, what requires a somewhat tedious
%study --that we don't reproduce here--, and show in particular that
%the domain $\{z: |y(z)|\leq 1, |y(z)|\leq 1\}$ is
%strictly included in $[z_{3}/2,z_{2}]\times
%[0,z_{1}/i]$ (recall from~(\ref{sp4-omega123}) that $z_1/\imath$ is
%real).

\begin{prop}
\label{sp4-continuation_h_h_tilde_covering}
 The functions $H(z)=h(x(z))$ and $\widetilde{H}(z)=\widetilde{h}(y(z))$
can be meromorphically continued from respectively
$\Lambda(-\pi/2,\pi/2)$ and $\Lambda(-\pi/2-\pi/n,\pi/2-\pi/n)$ up to
respectively $\mathbb{C}\setminus \Lambda(\pi,\pi)$ and
$\mathbb{C}\setminus \Lambda(\pi-\pi/n,\pi-\pi/n)$.
Moreover, these continuations satisfy
     \begin{equation}
     \label{sp4-stab_relations}
          H(z)=
          H(\xi(z)),\ \ \ \ \
          \widetilde{H}(z)=
          \widetilde{H}(\eta(z)),
          \ \ \ \ \ \forall z \in \mathbb{C},
     \end{equation}
and
     \begin{equation*}
          H(z)+\widetilde{H}(z)-
          x^{i_{0}}y^{j_{0}}(z)=
          \left\{\phantom{\begin{array}{ccccccc}
          \displaystyle 0 & \text{if} & z\notin\Lambda(\pi-\pi/n,\pi) \\
          \displaystyle -\sum_{w\in W_{n}}\left(-1\right)^{l\left(w\right)}
          x^{i_{0}}y^{j_{0}}\left(w(z)\right) & \text{if} & z\in \Lambda(\pi-\pi/n,\pi).
          \end{array}}\right.\phantom{ii}\phantom{iii}
     \end{equation*}

\vspace{-23mm}

     \begin{eqnarray}
          \phantom{Q\left(z\right)+\widetilde{Q}\left(z\right)-
          x^{i_{0}}y^{j_{0}}\left(z\right) }&&\hspace{19mm} 0 \ \
          \hspace{26mm} \textit{if}\ \   z\notin\Lambda(\pi-\pi/n,\pi)
          \phantom{iiiiii} \label{sp4-eH_s_s_functions_1}\\
          \phantom{Q\left(z\right)+\widetilde{Q}\left(z\right)-
          x^{i_{0}}y^{j_{0}}\left(z\right)}&&\hspace{-2mm} -\sum_{w\in W_{n}}\left(-1\right)^{l\left(w\right)}
          x^{i_{0}}y^{j_{0}}\left(w(z)\right)\ \ \hspace{-0.2mm} \textit{if}\ \ z\in \Lambda(\pi-\pi/n,\pi)
          \phantom{\mathbb{C}\setminus}
          \phantom{iiiiii} \label{sp4-eH_s_s_functions_2}
     \end{eqnarray}
where $l(w)$ is the length of $w$, i.e.\ the smallest $r$ for
which we can write $w=w_{1}\circ \cdots \circ w_{r}$, with $w_{1},\ldots,w_{r}$
equal to $\xi$ or $\eta$.
\end{prop}

\begin{rem}
\label{sp4-continuation_h_h_tilde}
As a consequence of Proposition~\ref{sp4-continuation_h_h_tilde_covering},
the generating functions $h$ and $\widetilde{h}$ can be continued as meromorphic
functions from the unit disc up to $\mathbb{C}\setminus [1,x_{4}]$ and
$\mathbb{C}\setminus [1,y_{4}]$ respectively.
Indeed, the formulas $h(x)=H(z)$ if $x(z)=x$ and $\widetilde{h}(y)=\widetilde{H}(z)$
if $y(z)=y$ define $h$ and $\widetilde{h}$ not ambiguously, thanks to~(\ref{sp4-stab_relations}).
Moreover, since we have $x(\Lambda(\pi,\pi))=[1,x_{4}]$ and
$y(\Lambda(\pi-\pi/n,\pi-\pi/n))=[1,y_{4}]$, see~(\ref{sp4-eH_cycles})
and Figure~\ref{sp4-transformation_cycles},
the previous formulas yield meromorphic continuations on
$\mathbb{C}\setminus [1,x_{4}]$ and $\mathbb{C}\setminus [1,y_{4}]$ respectively.
\end{rem}

\begin{proof}[Proof of Proposition~\ref{sp4-continuation_h_h_tilde_covering}]
To prove Proposition~\ref{sp4-continuation_h_h_tilde_covering}, we shall
heavily use the decomposition~(\ref{sp4-fundamental_domain_first_step}),
and precisely, we are going to define $H$ and $\widetilde{H}$ piecewise, by
defining them on each of the $2n$ domains $D$ that appear in the decomposition
(\ref{sp4-fundamental_domain_first_step}) to be equal to some functions $H_{D}$
and $\widetilde{H}_{D}$; it will then suffice to show that the functions $H$ and
$\widetilde{H}$ so-defined satisfy the conclusions of Proposition
\ref{sp4-continuation_h_h_tilde_covering}.

\medskip

$\ast$ In $D_{0}\subset \{z \in \mathbb{C}\cup\{\infty\}: |x(z)|\leq 1,
|y(z)|\leq 1\}$, we are going to use the most natural way
to define $H_{D_{0}}$ and $\widetilde{H}_{D_{0}}$, i.e.\ their power
series; so for $z\in D_{0}$ we set $H_{D_{0}}(z)=h(x(z))$ and
$\widetilde{H}_{D_{0}}(z)=\widetilde{h}(y(z))$.

\medskip

$\ast$ Then, for $k\in\{1,\ldots,n-1\}$, we define
$H_{D_{k}^{+}}$, $\widetilde{H}_{D_{k}^{+}}$ on $D_{k}^{+}$ and
$H_{D_{k}^{-}}$, $\widetilde{H}_{D_{k}^{-}}$ on $D_{k}^{-}$ by
     \begin{equation*}
          \left.\begin{array}{ccccccc}
          \forall z\in D_{k}^{+}=\xi\left(D_{k-1}^{-}\right)\hspace{-2mm} &: &
          H_{D_{k}^{+}}\left(z\right)=H_{D_{k-1}^{-}}\left(\xi\left(z\right)\right),
          & \widetilde{H}_{D_{k}^{+}}\left(z\right)=
          -H_{D_{k}^{+}}\left(z\right)+x^{i_{0}}y^{j_{0}}\left(z\right),\\
          \forall z\in D_{k}^{-}=\eta\left(D_{k-1}^{+}\right)\hspace{-2mm} &: &
          \widetilde{H}_{D_{k}^{-}}\left(z\right)=
          \widetilde{H}_{D_{k-1}^{+}}\left(\eta\left(z\right)\right), &
          H_{D_{k}^{-}}\left(z\right)=-\widetilde{H}_{D_{k}^{-}}\left(z\right)+
          x^{i_{0}}y^{j_{0}}\left(z\right).
          \end{array}\right.
     \end{equation*}

$\ast$ At last, for $z\in D_{n}$, we set
          $H_{D_{n}}(z)=H_{D_{n-1}^{-}}(\xi(z))$ and
          $\widetilde{H}_{D_{n}}(z)=\widetilde{H}_{D_{n-1}^{+}}(\eta(z))$.

\medskip

Therefore we have, for each of the $2n$ domains $D$ of
the decomposition~(\ref{sp4-fundamental_domain_first_step}),
defined two functions $H_{D}$ and $\widetilde{H}_{D}$. Then,
as said at the beginning of the proof, we set
$H(z)=H_{D}(z)$ and $\widetilde{H}(z)=\widetilde{H}_{D}(z)$
for all $z\in D$ and for all domains $D$ that appear
in~(\ref{sp4-fundamental_domain_first_step}).

\medskip

With this construction,~(\ref{sp4-stab_relations}) and~(\ref{sp4-eH_s_s_functions_1})
are immediately obtained. In order to prove (\ref{sp4-eH_s_s_functions_2}), we can
use the fact that it is possible to express \emph{all} the functions $H_{D}$,
$\widetilde{H}_{D}$ in terms of $H_{D_{0}}$, $\widetilde{H}_{D_{0}}$ and $x^{i_{0}}
y^{j_{0}}$ \textit{only}; we give, e.g., the expression of $H_{D_{2k}^{+}}$,
for any $2k\in\{1,\ldots ,n\}$:
     \begin{equation}
     \label{sp4-expression_H_2k_+}
          H_{D_{2k}^{+}}\left(z\right)=-\widetilde{H}_{D_{0}}
          \big((\eta\circ \xi)^{k}
          (z)\big)+\sum_{p=0}^{k-1}
          x^{i_{0}}y^{i_{0}}\big(\xi\circ (\eta\circ
          \xi)^{p}(z)\big)-\sum_{p=1}^{k-1}
          x^{i_{0}}y^{i_{0}}\big((\eta\circ \xi)^{p}(z)\big),
     \end{equation}
as well as that of $\widetilde{H}_{D_{2k}^{-}}$, for any $2k\in\{1,\ldots ,n\}$:
     \begin{equation}
     \label{sp4-expression_tilde_H_2k_-}
          \widetilde{H}_{D_{2k}^{-}}(z)=-H_{D_{0}}\big((\xi\circ \eta)^{k}(z)\big)+\sum_{p=0}^{k-1}
          x^{i_{0}}y^{i_{0}}\big(\eta\circ (\xi\circ \eta)^{p}(z)\big)-\sum_{p=1}^{k-1}
          x^{i_{0}}y^{i_{0}}\big((\xi\circ \eta)^{p}(z)\big).
      \end{equation}

As a consequence, we get Equation~(\ref{sp4-eH_s_s_functions_2}) for even values of $n$.
Indeed, for this it is actually sufficient first to add $H_{D_{n}}(z)$ and $\widetilde{H}_{D_{n}}(z)$, in other
words the equalities~(\ref{sp4-expression_H_2k_+}) and~(\ref{sp4-expression_tilde_H_2k_-}) above for $k=n/2$,
then to notice that if $n$ is even, $(\xi\circ \eta)^{n/2}=(\eta\circ\xi)^{n/2}$,
next to use that for $z\in D_{n}$, $H_{D_{0}}\big((\xi\circ \eta)^{n/2}(z)\big)+
\widetilde{H}_{D_{0}}\big((\xi\circ \eta)^{n/2}(z)\big)=
x^{i_{0}}y^{j_{0}}\big((\xi\circ \eta)^{n/2}(z)\big)$, see~(\ref{sp4-eH_s_s_functions_1}),
and finally to notice that $W_{n}$ equals
     \begin{equation*}
          \big\{1,\eta\xi,\ldots,(\eta\xi)^{n/2-1},
          \xi\eta,\ldots,(\xi\eta)^{n/2-1},\xi,\ldots,
          \xi(\eta\xi)^{n/2-1},\eta,\ldots,\eta(\xi\eta)^{n/2-1},
          (\xi\eta)^{n/2}\big\}.
     \end{equation*}

Likewise, we could write the expressions of
     \begin{equation*}
          H_{D_{2k}^{-}},\ \
          \widetilde{H}_{D_{2k}^{+}},\ \
          H_{D_{2k+1}^{+}},\ \
          \widetilde{H}_{D_{2k+1}^{+}},\ \
          H_{D_{2k+1}^{-}},\ \
          \widetilde{H}_{D_{2k+1}^{-}},
     \end{equation*}
    in terms  of $H_{D_{0}}$, $\widetilde{H}_{D_{0}}$
and $x^{i_{0}}y^{j_{0}}$
and we would verify that Equation~(\ref{sp4-eH_s_s_functions_2}) is still true for
odd $n$. Proposition~\ref{sp4-continuation_h_h_tilde_covering} is proven.
\end{proof}

\begin{rem}
\label{sp4-assumption_parameters_necessary}
We can now explain precisely why the assumption \textnormal{(H3)}
dealing with the values of the transition probabilities
is both natural and necessary for our study.

If we suppose \textnormal{(H1)} but
no more \textnormal{(H3)}, then the uniformization~(\ref{sp4-uniformization}) is the same, 
with $z_{0}=-[2p_{1,-1}]^{1/2}+\imath [2p_{1,0}]^{1/2}$. The transformations~(\ref{sp4-eH_cycles})
of the important cycles through the uniformization are also still valid and
the automorphism $\xi$ is yet again equal to $\xi(z)=$ $1/z$. As for $\eta$, it equals
$\eta(z)=z_{0}^{2}/z$; in particular, the group of the random walk $W=\langle \xi,\eta\rangle$
is finite if and only if there exists an integer $p$ such that $z_{0}^{2p}=1$. In this case,
if $n$ denotes the smallest of these positive integers $p$, the group $W$ is then of order $2n$.

If a such $n$ does not exist, then there is no hope to find a fundamental domain
for the action of the group $W$, neither to obtain any equality like~(\ref{sp4-eH_s_s_functions_2}).

If a such $n$ exists, then by using the fact that $z_{0}^{2n}=1$, in other
words the fact that $(-[2p_{1,-1}]^{1/2}+\imath [2p_{1,0}]^{1/2})^{2n}=1$, we
immediately obtain that $p_{1,0}=\sin(q\pi/n)^{2}/2$, for some integer $q$ having a
greatest common divisor with $n$ equal to $1$.

In this last case, we have $z_{0}=-\exp(-\imath q\pi/n)$; it is then easily proven
that the domain bounded by the cycles $x^{-1}([1,x_{4}])$ and $y^{-1}([1,y_{4}])$,
namely $\Lambda(\arg(z_{0}),\pi)=\Lambda(\pi-q\pi/n,\pi)$, is a fundamental domain
for the action of $W$ if and only if $q=1$. In particular, having an
equality like~(\ref{sp4-eH_s_s_functions_2}) is possible if and only
if $q=1$, see the proof of Proposition~\ref{sp4-continuation_h_h_tilde_covering}.

But it turns out that having an equality like~(\ref{sp4-eH_s_s_functions_2}) is essential
in what follows, particularly in Section~\ref{sp4-Martin_boundary}, where we have to
know very precisely the behavior of $H(z)+\widetilde{H}(z)-x^{i_{0}}y^{j_{0}}(z)$ near
$0$ and $\infty$. In the general case, the formulations of $H$ and 
$\widetilde{H}$---as solutions of boundary value problems---are so complex that we are not able
to pursue the analysis.

For all these reasons, we assume here that $p_{1,0}=\sin(\pi/n)^{2}/2$
for some integer $n$, in other words nothing else but \textnormal{(H3)}.
\end{rem}

\section{Harmonic functions}
\label{sp4-Harmonic_functions}

Section \ref{sp4-Harmonic_functions} aims at introducing and studying a certain
harmonic function associated with the process, which will be of the highest
importance in the forthcoming Section~\ref{sp4-Martin_boundary}.

%the fact of not reading the whole section but only
%Proposition~\ref{new_prop_properties_fn}, doesn't hinder the global
%understanding of the text.

It turns out that this harmonic function will be obtained
from the expansion near~$0$ of
$\sum_{w\in W_{n}}(-1)^{l(w)}x^{i_{0}}y^{j_{0}}(w(z))$,
quantity which is appeared in~(\ref{sp4-eH_s_s_functions_2});
%of Proposition~\ref{sp4-continuation_h_h_tilde_covering}
 this is why
we begin here by studying closely the behavior of the latter sum
in the neighborhood of $0$.

Note first that thanks to the expression (\ref{sp4-def_automorphisms_C_s})
of the automorphisms $\xi$ and $\eta$, we have
     \begin{equation}
     \label{sp4-quantity_i0_j0}
          \sum_{w\in W_{n}}\left(-1\right)^{l(w)}
          x^{i_{0}}y^{j_{0}}\left(w\left(z\right)\right)=
          \sum_{k=0}^{n-1}\left[
          x^{i_{0}}y^{i_{0}}\left(\exp(-2\imath k \pi/n)z\right)-
          x^{i_{0}}y^{i_{0}}\left(\exp(-2\imath k \pi/n)/z\right)\right].
     \end{equation}

Let us now take the following notations for the expansion at $0$
of the function $x^{i_{0}}y^{j_{0}}$:
     \begin{equation}
     \label{sp4-notation_x_i0_y_j0}
          x^{i_{0}}y^{i_{0}}\left(z\right)=
          \sum_{p\geq 0}\kappa_{p}
          \left(i_{0},j_{0}\right) z^{p},
     \end{equation}
and notice that with~(\ref{sp4-uniformization}), we obtain
that for $z$ close to $0$,
     \begin{equation}
     \label{sp4-expansion_x_i0_y_j0_infinity}
          x^{i_{0}}y^{i_{0}}\left(1/z\right)=
          \sum_{p\geq 0}\overline{\kappa_{p}}
          \left(i_{0},j_{0}\right) z^{p}.
     \end{equation}

In a general setting, if $f$ is holomorphic
in a neighborhood of $0$ with expansion $f(z)=\sum_{p\geq 0}f_{p}z^{p}$,
then $\sum_{k=0}^{n-1}f(\exp(-2\imath k \pi/n)z) =
\sum_{k=0}^{n-1}f(\exp( 2\imath k \pi/n)z) =
\sum_{p\geq 0} n f_{n p}z^{n p}$.

This is why, by using~(\ref{sp4-notation_x_i0_y_j0})
and~(\ref{sp4-expansion_x_i0_y_j0_infinity}), we obtain
that the sum~(\ref{sp4-quantity_i0_j0}) is equal to
     \begin{equation}
     \label{sp4-s_s_expansion_x_i0_y_j0}
          \sum_{w\in W_{n}}\left(-1\right)^{l(w)}
          x^{i_{0}}y^{j_{0}}\left(w\left(z\right)\right)=
          \sum_{p\geq 1} n \left[\kappa_{n p}
          \left(i_{0},j_{0}\right)-\overline{\kappa_{n p}}
          \left(i_{0},j_{0}\right)\right] z^{n p}.
     \end{equation}

We are now going to be interested in the term corresponding to $p=1$
in the sum~(\ref{sp4-s_s_expansion_x_i0_y_j0}), and we set
%     \begin{equation}
%     \label{sp4-def_f_n}
%          f_{n}\left(i_{0},j_{0}\right)=n\left[\kappa_{n}\left(i_{0},j_{0}\right)-
%          \overline{\kappa_{n}}\left(i_{0},j_{0}\right)\right]/
%          \left[\left(-1\right)^{n}\imath \right].
%     \end{equation}
     \begin{equation}
     \label{sp4-def_f_n}
          f_{n}\left(i_{0},j_{0}\right)=n\frac{\kappa_{n}\left(i_{0},j_{0}\right)-
          \overline{\kappa_{n}}\left(i_{0},j_{0}\right)}
          {\left(-1\right)^{n}\imath}.
     \end{equation}

We are now going to prove Proposition \ref{new_prop_properties_fn},
but before we study some of its consequences.

\begin{cor}
\label{sp4-never_hit}
The Doob $f_{n}$-transform process of $(X,Y)$ never hits
the boundary.
\end{cor}

\begin{rem}
The function $f_{n}$ defined in (\ref{sp4-def_f_n}) is quite explicit.
Indeed, by using the Cauchy product, $\kappa_{n}$ (and therefore also
$f_{n}$ via (\ref{sp4-def_f_n})) can be written in terms of the coefficients of the expansions
of $x$ and $y$ at $0$, and these coefficients are easily calculated,
see Equation~(\ref{sp4-explicit_expansions_x_y}).~We don't know the factorized expression of $f_{n}$ for general values of $n$.
However, we have given in Subsection \ref{sub_harmo_bm}, as examples, the factorized forms of $f_{3}$, $f_{4}$ and $f_{6}$.
\end{rem}

\begin{rem}
The explicit expression and the harmonicity of the function $f_{4}$
have already been obtained by Biane in~\cite{Bi3}.   

The quantity $f_{4}(i_{0},j_{0})$ also appears as a multiplicative factor
in the asymptotic tail~dis- tribution of the hitting time of the boundary of
$\mathbb{Z}_{+}^{2}$ for the process $(X,Y)$ associated~with $n=4$ and starting
from the initial state $(i_{0},j_{0})$. Indeed, in \cite{DO,mythesis},
denoting by $\tau=\inf\{ k\geq 0: X(k)=0$ or $Y(k)=0\}$, it is proven that
$\mathbb{P}_{(i_{0},j_{0})}[\tau>k]\sim C f_{4}(i_{0},j_{0})/k^{2}$, where $C>0$.

%the authors of~\cite{DO} make explicit
%$\mathbb{P}_{(i_{0},j_{0})}[\tau>k]$ in the case $n=4$, by using an extension of the well-known reflection principle. It
%is then possible to deduce that the following asymptotic holds: $\mathbb{P}_{(i_{0},j_{0})}[\tau>k]
%\sim C f_{4}(i_{0},j_{0})/k^{2}$, $C$ being some positive constant.

In particular, we can specify Corollary~\ref{sp4-never_hit} in the case $n=4$.
Indeed, using the following equality for $l<k$ (obtained from the strong
Markov property of the process $(X,Y)$):
     \begin{equation*}
          \mathbb{P}_{\left(i_{0},j_{0}\right)}
          \left[\left(X\left(l\right),Y\left(l\right)\right)=
          \left(i,j\right) | \tau>k\right]=
          \mathbb{P}_{\left(i_{0},j_{0}\right)}
          \left[\left(X\left(l\right),Y\left(l\right)\right)=
          \left(i,j\right)\right]
          \frac{\mathbb{P}_{\left(i,j\right)}\left[\tau>k-l\right]}
          {\mathbb{P}_{\left(i_{0},j_{0}\right)}\left[\tau>k\right]},
     \end{equation*}
the asymptotic of \cite{DO,mythesis} yields that the Doob
$f_{4}$-transform process is equal in distribution to the limit as
$k\to\infty$ of the process conditioned on the event
$[\tau>k]$.
\end{rem}

\begin{rem}
Let $n\geq 3$. Proposition~\ref{new_prop_properties_fn} gives that
there exists at least one positive~harmonic function for the
process $(X,Y)$. Corollary~\ref{sp4-corollary_Martin_boundary}
entails that $f_{n}$ is in fact
the unique---up to the positive multiplicative constants---positive harmonic function for  $(X,Y)$.
\end{rem}

\begin{proof}[Proof of Proposition~\ref{new_prop_properties_fn}.]
The fact that $f_{n}$ takes real values is immediate from its definition.
For the rest of the proof of~\ref{sp4-it_polynomial}, we are going to use the
following straightforward fact: for any $f(z)=1+\sum_{p\geq 1}f_{p,1}z^{p}$,
note $1+\sum_{p\geq 1}f_{p,i}z^{p}$ the expansion at $0$ of
$f(z)^{i}$; then $f_{p,i}$ is a polynomial of
degree equal or less than $p$ in $i$, with dominant term
equal to $f_{1,1}^{p} i^{p}/p !$. In particular,
$f_{p,i}$ is of degree exactly $p$ if and only if $f_{1,1}\neq 0$.

In our case, it is immediate from~(\ref{sp4-uniformization})
that $\kappa_{1}(1,0)=-4\cos(\pi/n)\neq 0$ and~$\kappa_{1}(0,1)
=-4\exp(\imath \pi/n)\neq 0$. This is why, for any non-negative integer $p$,
$\kappa_{p}(i,0)$ is a polynomial of degree $p$ in $i$;
likewise, $\kappa_{p}(0,j)$ is a polynomial of degree $p$ in $j$.
In particular,~$\kappa_{n}(i,j)=$ $\sum_{p=0}^{n}\kappa_{p}(0,j)\kappa_{n-p}(i,0)$
is a polynomial in $i,j$ of degree $n$, with dominant term equal to
     \begin{equation*}
          \sum_{p=0}^{n}\frac{\kappa_{1}(0,1)^{p}}{p !}j^{p}
          \frac{\kappa_{1}(1,0)^{n-p}}{(n-p) ! }i^{n-p}.
     \end{equation*}
In this way, we obtain that
$f_{n}$ is a polynomial in $i,j$ of degree $n$, with
dominant term equal to (after simplification)
          \begin{equation}
          \label{sp4_future_dominant}
               \frac{2^{2n+1}}{(n-1)!}
               \sum_{p=1}^{n-1}C_{n}^{p} \sin(p\pi/n)\cos(\pi/n)^{n-p}
               j^{p} i^{n-p}.
          \end{equation}
Assertion \ref{sp4-it_polynomial} follows then immediately.

To prove~\ref{sp4-it_harmo}, it is enough to
show that $\kappa_{n}$
%, defined in~(\ref{sp4-notation_x_i0_y_j0}),
is harmonic. To show that, start by using the obvious equality
%$Q(x(z),y(z))=0$. As an immediate consequence,
$x^{i-1}y^{j-1}(z)Q(x(z),y(z))=0$, which reads
$x^{i}y^{j}(z)=p_{1,0}x^{i+1}y^{j}(z)+
p_{1,0}x^{i-1}y^{j}(z)+p_{1,-1}x^{i+1}y^{j-1}(z)
+p_{1,-1}x^{i-1}y^{j+1}(z)$.
Then, (\ref{sp4-notation_x_i0_y_j0}) yields~that
     \begin{equation*}
          \sum_{p\geq 0}\left[
          \kappa_{p}(i,j)\hspace{-0.7mm}-\hspace{-0.7mm}p_{1,0}\kappa_{p}(i\hspace{-0.7mm}+\hspace{-0.7mm}1,j)
          \hspace{-0.7mm}-\hspace{-0.7mm}p_{1,0}\kappa_{p}(i\hspace{-0.7mm}-\hspace{-0.7mm}1,j)\hspace{-0.7mm}-\hspace{-0.7mm}
          p_{1,-1}\kappa_{p}(i\hspace{-0.7mm}+\hspace{-0.7mm}1,j\hspace{-0.7mm}-\hspace{-0.7mm}1)
          \hspace{-0.7mm}-\hspace{-0.7mm}p_{1,-1}\kappa_{p}(i\hspace{-0.7mm}-\hspace{-0.7mm}1,j\hspace{-0.7mm}+\hspace{-0.7mm}1)\right]z^{p}
     \end{equation*}
is identically zero; this means that all the $\kappa_{p}$, $p\geq 0$
are harmonic, hence in particular $\kappa_{n}$.

%In order to prove~\ref{sp4-it_value_1_1}, start by using the
%equalities
%     \begin{equation}
%     \label{sp4-explicit_expansions_x_y}
%          x(z)=1+\frac{4}{\tan(\pi/n)}
%          \sum_{p\geq 1}\left(-1\right)^{p}
%          \sin\left(p\pi/n\right)z^{p},\ \ \ \ \
%          y(z)=1+4\sum_{p\geq 1}\left(-1\right)^{p}
%          p\exp\left(\imath p\pi/n\right)z^{p},
%     \end{equation}
%which are easily obtained from~(\ref{sp4-uniformization}).
%Then, thanks to Cauchy's product of these two series, it is possible
%to write the explicit expression of the $\kappa_{p}(1,1)$, $p\geq 0$,
%and finally to obtain assertion~\ref{sp4-it_value_1_1}.
%     \begin{equation*}
%          \kappa_{n}\left(1,1\right)=
%          4\left(-1\right)^{n+1}+
%          \frac{4\imath \left(-1\right)^{n}}
%          {\tan\left(\pi/n\right)}\left(n^{2}
%          +\frac{\imath n}{\tan\left(\pi/n\right)}\right).
%     \end{equation*}

In order to prove~\ref{sp4-it_zero_boundary}, we need to know explicitly
the expansions of $x$ and $y$ at $0$. From (\ref{sp4-uniformization}), we
immediately obtain these expansions:
     \begin{equation}
     \label{sp4-explicit_expansions_x_y}
          x(z)=1+\frac{4}{\tan(\pi/n)}
          \sum_{p\geq 1}\left(-1\right)^{p}
          \sin\left(p\pi/n\right)z^{p},\ \ \ \
          y(z)=1+4\sum_{p\geq 1}\left(-1\right)^{p}
          p\exp\left(\imath p\pi/n\right)z^{p}.
     \end{equation}

We show now the first part of~\ref{sp4-it_zero_boundary}, namely the fact
that $f_{n}(i,0)=0$ for all non-negative integer $i$.
As it can be noticed from~(\ref{sp4-explicit_expansions_x_y}), the
coefficients of $x$ are real. For this reason, for
all integers $i$ and $p$,
$\kappa_{p}(i,0)$ is also real and thus $f_{p}(i,0)=0$;
in particular, $f_{n}(i,0)=0$.

As for the second part of~\ref{sp4-it_zero_boundary}, namely the fact
that for all $j\geq 0$, $f_{n}(0,j)=0$, we prove
that $\kappa_{n}(0,j)$ is real---however, it isn't true that
for all $j$ and $p$, $\kappa_{p}(0,j)$ is real.

In order to obtain $\kappa_{p}(0,j)$---that is, the $p$th coefficient of the Taylor series of $y(z)^{j}$---we add all the terms of the form $\kappa_{p_{1}}(0,1)\kappa_{p_{2}}(0,1)\times\cdots \times\kappa_{p_{j}}(0,1)$
with $p_{1}+\cdots +p_{j}=p$, this is nothing else but the Cauchy's product
of the $j$ series $y(z)$.
In other words, using~(\ref{sp4-explicit_expansions_x_y}),
we add terms of the form $p_{1}\times \cdots\times p_{j}(-1)^{p_{1}+\cdots +p_{j}}
\exp(\imath [p_{1}+\cdots +p_{j}]\pi/n)$. As a consequence, $\kappa_{p}(0,j)$ can be written as
$\varphi_{p}(j)(-1)^{p}\exp(\imath p\pi/n)$, with $\varphi_{p}(j)>0$ if $j>0$.

In the particular case $p=n$, we obtain
$\kappa_{n}(0,j)=-\varphi_{n}(j)(-1)^{n}$;
$\kappa_{n}(0,j)$ is therefore real and, immediately, $f_{n}(0,j)=0$.

We prove now~\ref{sp4-it_positive}.
With~(\ref{sp4-explicit_expansions_x_y}),
it is clear that the sequence $\kappa_{0}(1,0),\ldots ,\kappa_{n-1}(1,0)$ is
alternating, in the sense that for all $p\in\{0,\ldots ,n-1\}$, $(-1)^{p}\kappa_{p}(1,0)>0$.
In particular, it follows from general results on power series that the sequence
$\kappa_{0}(i,0),\ldots ,\kappa_{n-1}(i,0)$ is still alternating, for any $i>0$.
%--~on the other hand, note that $\kappa_{0}(i,0),\ldots ,\kappa_{n}(i,0)$ is
%{a priori} not alternating

In addition, by using the Cauchy's product of $x(z)^{i}$ and $y(z)^{j}$, we obtain
that~$
          \kappa_{n}(i,j)=
          \kappa_{n}(i,0)+\kappa_{n}(0,j)+
          \sum_{p=1}^{n-1}(-1)^{p}
          \exp(\imath p\pi/n)\varphi_{p}(j)
          \kappa_{n-p}(i,0)$.
Then, by definition of $f_{n}(i,j)$ and by using the fact that $\kappa_{n}(i,0)$
and $\kappa_{n}(0,j)$ are real, we get
     \begin{equation*}
          f_{n}(i,j)=2 n \left(-1\right)^{n}
          \sum_{p=1}^{n-1}\left(-1\right)^{p}
          \sin\left(p\pi/n\right)\varphi_{p}(j)
          \kappa_{n-p}(i,0).
     \end{equation*}
     But we have already proven that $\varphi_{p}(j)>0$ if $j>0$
and that $(-1)^{n-p}\kappa_{n-p}(i,0)>0$ if $i>0$; above, $f_{n}$ is
thus written as the sum of $n-1$ positive terms,
and is, therefore, positive.
\end{proof}

\begin{proof}[Proof of Proposition~\ref{prop_link}]
Recall that $h(\rho\exp(\imath \theta))
=\rho^{n}\sin(n\theta)$, see Subsection \ref{sub_harmo_bm}.
 Setting $u=\rho \cos(\theta)$ and
 $v=\rho \sin(\theta)$ then yields:
     \begin{equation*}
          h(u,v)=\sum_{p=0}^{(n-1)/2}C_{n}^{2p+1}
          (-1)^{p}u^{n-(2p+1)}v^{2p+1}.
     \end{equation*}
Next we easily check that up to a multiplicative constant,
$h(\phi(i_0,j_0))$ equals the~dominant term of $f_{n}(i_0,j_0)$,
namely $(2^{2n+1}/(n-1)!)
          \sum_{p=1}^{n-1}C_{n}^{p} \sin(p\pi/n)\cos(\pi/n)^{n-p}
          i_0^{p} j_0^{n-p}$, see~(\ref{sp4_future_dominant}).
\end{proof}

\begin{proof}[Proof of Proposition~\ref{prop_non_homo}]
Since $f_n$ is a polynomial of degree exactly $n$,
see \ref{sp4-it_polynomial} of Proposition~\ref{new_prop_properties_fn},
it is clearly enough to prove that for $n\geq 5$, $f_n(2,2)/f_n(1,1)\neq 2^n$.
For this we shall find both $f_n(1,1)$ and $f_n(2,2)$ in terms of $n$,
it will then be manifest that $f_n(2,2)/f_n(1,1)\neq 2^n$.

Let us first prove that $f_n(1,1)=8n^{3}/\tan(\pi/n)$. The Cauchy product
of $x(z)$ by $y(z)$ and the use of (\ref{sp4-explicit_expansions_x_y}) yields
     \begin{equation*}
          \kappa_{n}(1,1)=4(-1)^{n+1}n+[16(-1)^n/\tan(\pi/n)]\sum_{p=1}^{n-1}
          p\exp(\imath p \pi/n)\sin((n-p)\pi/n).
     \end{equation*}
Using then the exponential expression of $\sin((n-p)\pi/n)$ as well as
the identity $\sum_{p=1}^{n-1}px^{p}=[nx^n(x-1)-x(x^n-1)]/[(x-1)^2]$
applied to $x=\exp(2\imath \pi/n)$ entails:
     \begin{equation*}
          \kappa_{n}(1,1)=4(-1)^{n+1}n+[8(-1)^n/(\imath \tan(\pi/n))][-n(n-1)/2+n
          \exp(-\imath \pi/n)/(2\imath \sin(\pi/n))].
     \end{equation*}
Using (\ref{sp4-def_f_n}) finally gives the announced value of $f_n(1,1)$.

Thanks to calculations of the same kind than for $f_n(1,1)$, we reach
the conclusion that $f_n(2,2)=(16/3)n^3[n^2+2-6/\tan(\pi/n)^2]/[\sin(\pi/n)^2\tan(\pi/n)]$. Then it becomes obvious that
for $n\geq 5$, $f_n(2,2)/f_n(1,1)\neq 2^n$, which completes the proof of
Proposition~\ref{prop_non_homo}.
\end{proof}

\section{Asymptotic of the Green functions}
\label{sp4-Martin_boundary}

\textit{\textbf{{Sketch of the proof of Theorem~\ref{sp4-Main_theorem_Green_functions}.}}}
We shall begin by expressing $G_{i,j}$
as a double integral, using for this Cauchy's formulas and
(\ref{sp4-functional_equation}), see Equation~(\ref{sp4-application_Cauchy_formula}). Then we will make the change of variable
given by the uniformization~(\ref{sp4-uniformization}) and we will apply the residue
theorem; in this way, we will write $G_{i,j}$ as the sum $G_{i,j,1}+G_{i,j,2}$
of two single integrals w.r.t.\ the uniformization variable but on two contours {a priori}
different, see~(\ref{sp4-n_f_G_i_j_1}) and (\ref{sp4-n_f_G_i_j_2}). Then
we will show, using Cauchy's theorem and
Proposition \ref{sp4-continuation_h_h_tilde_covering}, that it is possible
to move these contours of integration until having the same contours
for both integrals $G_{i,j,1}$ and $G_{i,j,2}$. Finally, by 
using~(\ref{sp4-eH_s_s_functions_2})~we will obtain~(\ref{sp4-final_formulation_G_i_j}), which is the most important explicit formulation of the
$G_{i,j}$, starting from which we will get their asymptotic. In~(\ref{sp4-final_formulation_G_i_j}),
$G_{i,j}$ will be written as an integral on the contour $\exp(\imath \theta)\mathbb{R}_{+}\cup\{\infty\}$,
for some $\theta\in[\pi-\pi/n,\pi]$.

After having chosen an appropriate value of $\theta\in[\pi-\pi/n,\pi]$, see~(\ref{sp4-def_rho_j/i}),
we will see~that this is quite normal to decompose the contour into three
parts, namely a neighborhood of $0$, one of $\infty$ and an intermediate part.
Indeed, the function $x(z)^{i} y(z)^{j}$ that appears in the integrand
of~(\ref{sp4-final_formulation_G_i_j}) is, on the contour $\exp(\imath \theta)
\mathbb{R}_{+}\cup\{\infty\}$, close to $1$ near $0,\infty$ and strictly
larger than $1$ elsewhere. Next, we will study successively these contributions in three paragraphs,
using for this essentially the Laplace's method,
what will conclude the proof of Theorem \ref{sp4-Main_theorem_Green_functions}.

\medskip

\medskip

{\raggedright \textit{\textbf{Beginning of the proof of Theorem~\ref{sp4-Main_theorem_Green_functions}.}}}
Equation~(\ref{sp4-functional_equation}) yields immediately that the gene- rating function
$G$ of the Green functions is holomorphic in $\{(x,y)\in \mathbb{C}^{2}: |x|<1, |y|<1\}$.
As a consequence and using again Equation~(\ref{sp4-functional_equation}), the Cauchy's
formulas allow us to write its coefficients $G_{i,j}$ as the following double integrals:
     \begin{equation}
     \label{sp4-application_Cauchy_formula}
          G_{i,j}=\frac{1}{\left[2\pi \imath \right]^{2}}
          \iint_{\substack{\left|x\right|=1 \\\left|y\right|=1 }}
          \frac{G\left(x,y\right)}{x^{i}y^{j}}\text{d}x\text{d}y
          = \frac{1}{\left[2\pi \imath \right]^{2}}
          \iint_{\substack{\left|x\right|=1 \\\left|y\right|=1 }}
          \frac{h\left(x\right)+\widetilde{h}
          \left(y\right)-x^{i_{0}}y^{j_{0}}}
          {x^{i}y^{j}Q\left(x,y\right)}\text{d}x\text{d}y,
     \end{equation}
where the circles $\{|x|=1\}=\{|y|=1\}=\{\exp(\imath \theta): \theta\in[0,2\pi[\}$
are orientated according~to the increasing values of $\theta$.

With~(\ref{sp4-application_Cauchy_formula}), we can thus write $G_{i,j}$ as the sum
$G_{i,j}=G_{i,j,1}+G_{i,j,2}$, where
      \begin{align*}
          G_{i,j,1}&= \frac{1}{\left[2\pi \imath \right]^{2}}
          \int_{\left|x\right|=1}
          \frac{h\left(x\right)}{x^{i}}\int_{\left|y\right|=1}
          \frac{\text{d}y}{y^{j}Q\left(x,y\right)}\text{d}x,
          \phantom{+
          \frac{1}{\left[2\pi \imath \right]^{2}}
          \int_{\left|y\right|=1}
          \frac{1}{y^{j-j_{0}}}\int_{\left|x\right|=1}
          \frac{\text{d}x}{x^{i-i_{0}}Q\left(x,y\right)}\text{d}y\hspace{0.6mm}} \\
          G_{i,j,2} &=\frac{1}{\left[2\pi \imath \right]^{2}}
          \int_{\left|y\right|=1}
          \frac{\widetilde{h}\left(y\right)}{y^{j}}\int_{\left|x\right|=1}
          \frac{\text{d}x}{x^{i}Q\left(x,y\right)}\text{d}y+
          \frac{1}{\left[2\pi \imath \right]^{2}}
          \int_{\left|y\right|=1}
          \frac{1}{y^{j-j_{0}}}\int_{\left|x\right|=1}
          \frac{\text{d}x}{x^{i-i_{0}}Q\left(x,y\right)}\text{d}y.
     \end{align*}

We are now going to make in $G_{i,j,1}$ the change of variable $x=x(z)$.
For this, we notice that if $\Lambda(\pi/2,\pi/2)=\{\imath t: t\in [0,\infty]\}$
is orientated according to increasing values of $t$, then the
equality $x(\Lambda(\pi/2,\pi/2))=-\{|x|=1\}$ holds---in the sense of the orientated contours---,
see (\ref{sp4-eH_cycles}) and Figure~\ref{sp4-transformation_cycles}. In this way and
using in addition the identity $q(x(z))=H(z)$, we get
     \begin{equation*}
          G_{i,j,1}= -\frac{1}{\left[2\pi \imath \right]^{2}}
          \int_{\Lambda(\pi/2,\pi/2)}
          \frac{H\left(z\right)}{x(z)^{i}}\int_{\left|y\right|=1}
          \frac{\text{d}y}{y^{j}Q\left(x(z),y\right)}x'(z)\text{d}z.
     \end{equation*}
     But $Q(x(z),y)=0$ if and only if $y\in\{y(z),x(z)^{2}/y(z)\}$,
see~(\ref{sp4-characterization_automorphisms}). Moreover,
if~$z$~belongs to $\Lambda(\pi/2,\pi/2)\setminus \{0,\infty\}$, then
$|y(z)|>1$, see Figure~\ref{sp4-transformation_cycles}. The residue
theorem at infinity therefore entails that for such $z$,
$\int_{|y|=1}\text{d}y/[y^{j}Q(x(z),y)]=-2\pi \imath /
[y(z)^{j}\partial_{y}Q(x(z),y(z))]$.
Finally, we have proven that
     \begin{equation}
     \label{sp4-n_f_G_i_j_1}
          G_{i,j,1} = \frac{1}{2\pi \imath }\int_{\Lambda(\pi/2,\pi/2)}
          \frac{H\left(z\right)}
          {x\left(z\right)^{i}
          y\left(z\right)^{j}}
          \frac{x'\left(z\right)}
          {\partial_{y}Q(x(z),y(z))}\text{d}z.
     \end{equation}

A similar reasoning yields
     \begin{equation}
     \label{sp4-n_f_G_i_j_2}
          G_{i,j,2}=-\frac{1}{2\pi \imath }
          \int_{\Lambda(-\pi/2-\pi/n,-\pi/2-\pi/n)}
          \frac{\widetilde{H}\left(z\right)-
          x\left(z\right)^{i_{0}}y\left(z\right)^{j_{0}}}
          {x\left(z\right)^{i}y\left(z\right)^{j}}
          \frac{y'\left(z\right)}
          {\partial_{x}Q(x(z),y(z))}\text{d}z.
     \end{equation}

We are now going to explain why it is possible
to move the contours of integration of
both integrals~(\ref{sp4-n_f_G_i_j_1})
and~(\ref{sp4-n_f_G_i_j_2})
up to $\Lambda(\theta,\theta)$, for any
$\theta\in[\pi-\pi/n,\pi]$---see Figure~\ref{sp4-choice_contour} below.

Start by considering $G_{i,j,1}$ in~(\ref{sp4-n_f_G_i_j_1}).
Thanks to Cauchy's theorem, it is sufficient
to show that the integrand of $G_{i,j,1}$ is
holomorphic inside of $\Lambda(\pi/2,\pi)$,
domain which is
horizontally hatched on Figure~\ref{sp4-choice_contour};
let us thus prove this fact.

On one hand, with~(\ref{sp4-def_Q})
and~(\ref{sp4-uniformization}), on the domain $\Lambda(\pi/2,\pi)$ we get $x'(z)/\partial_{y}Q(x(z),y(z))=
-\imath/(2[p_{1,0} p_{1,-1}]^{1/2}z)$,
and the latter function has manifestly no pole.
On the other hand, it is possible to deduce from
the proof of Proposition~\ref{sp4-continuation_h_h_tilde_covering}
that the only poles of $H$ are at $z_{0}$ and $\overline{z_{0}}$.
In particular, using~(\ref{sp4-uniformization}), we obtain that
for $i$ or $j$ large enough, $H(z)/[x(z)^{i}y(z)^{j}]$
has no pole in $\Lambda(\pi/2,\pi)$. Therefore,
for $i$ or $j$ large enough, the integrand of
$G_{i,j,1}$ has no pole in $\Lambda(\pi/2,\pi)$
and we can thus move the contour from
$\Lambda(\pi/2,\pi/2)$ to $\Lambda(\theta,\theta)$,
for any $\theta\in [\pi/2,\pi]$. Note that it isn't possible
to move the contour beyond $\Lambda(\pi,\pi)$, since $\Lambda(\pi,\pi)$
is a singular curve for $H$---indeed, remember that
$\Lambda(\pi,\pi)=x^{-1}([1,x_{4}])$ and see
Proposition~\ref{sp4-continuation_h_h_tilde_covering}.

By similar considerations, we obtain that it is possible
to move the initial contour of integration of $G_{i,j,2}$ up
to $\Lambda(\theta,\theta)$, for any $\theta\in [\pi-\pi/n,3\pi/2-\pi/n]$.

In particular, if we wish to have the same contour of
integration for $G_{i,j,1}$ and $G_{i,j,2}$, we can choose
$\Lambda(\theta,\theta)$, for any $\theta\in[\pi-\pi/n,\pi]
=[\pi/2,\pi]\cap [\pi-\pi/n,3\pi/2-\pi/n]$.

 \begin{figure}[!ht]
 \begin{center}
 \begin{picture}(290.00,170.00)
 \includegraphics{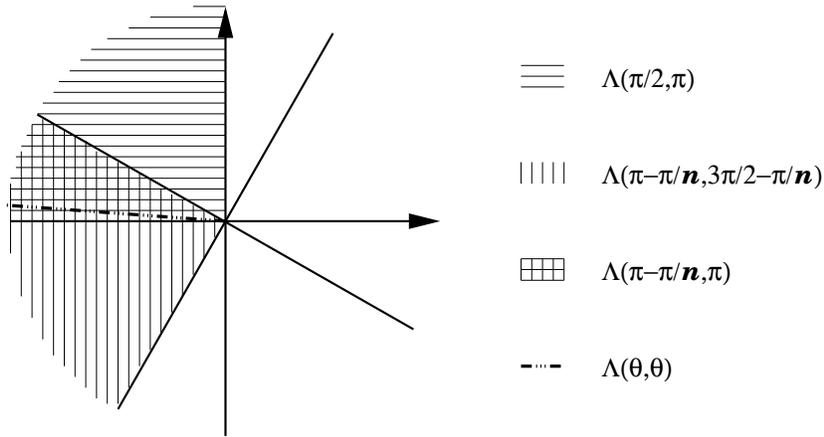}
 \end{picture}
 \end{center}
 \caption{Change of the contours of integration in the
          integrals~(\ref{sp4-n_f_G_i_j_1}) and~(\ref{sp4-n_f_G_i_j_2})}
 \label{sp4-choice_contour}
 \end{figure}

Using then the equality $x'(z)/\partial_{y}Q(x(z),y(z))=-
y'(z)/\partial_{x}Q(x(z),y(z))$,
that comes from differentiating $Q(x(z),y(z))=0$, as well as
(\ref{sp4-n_f_G_i_j_1}),~(\ref{sp4-n_f_G_i_j_2}) and~(\ref{sp4-eH_s_s_functions_2})---we can use (\ref{sp4-eH_s_s_functions_2}) since $\theta\in [\pi-\pi/n,\pi]$
and thus $\Lambda(\theta,\theta)\subset\Lambda(\pi-\pi/n,\pi)$---,
we obtain the following final explicit formulation for $G_{i,j}$, $\theta$ being
any angle in $[\pi-\pi/n,\pi]$:
     \begin{equation}
     \label{sp4-final_formulation_G_i_j}
          G_{i,j}=\frac{1}{4\pi [p_{1,0} p_{1,-1}]^{1/2} }
          \int_{\Lambda\left(\theta,\theta\right)}
          \left[\frac{1}{z}\sum_{w\in W_{n}}\left(-1\right)^{l(w)}
          x^{i_{0}}y^{j_{0}}\left(w\left(z\right)\right)\right]
          \frac{1}{x\left(z\right)^{i}
          y\left(z\right)^{j}}\text{d}z.
     \end{equation}

On the contour $\Lambda(\theta,\theta)\subset \Lambda(\pi-\pi/n,\pi)$,
the modulus of the function $x(z)^{i}y(z)^{j}$ is larger than $1$, see
Figure~\ref{sp4-transformation_cycles}. Moreover, it goes to $1$
when and only when $z$ goes to $0$ or to $\infty$. This is why it
seems normal to decompose the contour $\Lambda(\theta,\theta)$ into a part near $0$, another near $\infty$
and the residual part, and to think that the parts near $0$ and $\infty$ will lead
to the asymptotic of $G_{i,j}$, while the residual part will lead to a
negligible contribution. But how to find the best contour in order to achieve this idea?
In other words, how to find the value of $\theta\in[\pi-\pi/n,\pi]$ for which the calculation of the
asymptotic of~(\ref{sp4-final_formulation_G_i_j}) will be the easiest?

For this, we are going to consider with details the function $x(z)^{i}y(z)^{j}$,
or equivalently
     \begin{equation*}
          \chi_{j/i}(z)=\ln(x(z))+(j/i)\ln(y(z)).
     \end{equation*}

Incidentally this is why, from now on, we suppose that $j/i\in[0,M]$, for some
$M<\infty$. Indeed, the function $\chi_{j/i}$ is manifestly not adapted
to the values $j/i$ going to $\infty$; for~such $j/i$, we will consider, later,
the function $(i/j)\chi_{j/i}(z)=(i/j)\ln(x(z))+\ln(y(z))$. Nevertheless,
$M$ can be so large as wished and in what follows, we assume that some $M>0$ is~fixed.

With~(\ref{sp4-uniformization}), we easily obtain the explicit expansion of
$\chi_{j/i}$ at $0$:
     \begin{equation}
     \label{sp4-expansion_chi_0}
          \chi_{j/i}\left(z\right)=\sum_{p\geq 0}
          \nu_{2p+1}\left(j/i\right)z^{2p+1},\ \ \ \ \
          \nu_{2p+1}\left(j/i\right)=\frac{2}{2p+1}
          \left[{z_{0}}^{2p+1}+\overline{z_{0}}^{2p+1}
          +2\left(j/i\right)\overline{z_{0}}^{2p+1}\right].
     \end{equation}
Likewise, again with~(\ref{sp4-uniformization}), we get that for $z$
near $\infty$, $\chi_{j/i}(z)=\sum_{p=0}^{\infty}
\overline{\nu_{2p+1}}(j/i)1/z^{2p+1}$.

Consider now the steepest descent path associated with
$\chi_{j/i}$, in other words the function $z_{j/i}(t)$
defined by $\chi_{j/i}(z_{j/i}(t))=t$. By inverting
the latter equality, we easily obtain that the half-line
%$(1/[2({z_{0}}+\overline{z_{0}}+2(j/i)\overline{z_{0}})])\mathbb{R}_{+}$
$(1/\nu_{1}(j/i))\mathbb{R}_{+}\cup\{\infty\}$ is tangent at $0$ and at $\infty$
to this steepest descent path.

Let us now set
     \begin{equation}
     \label{sp4-def_rho_j/i}
          \rho_{j/i}=1/\nu_{1}\left(j/i\right)=
          1/\left[2\left({z_{0}}+\overline{z_{0}}+
          2\left(j/i\right)\overline{z_{0}}\right)\right].
     \end{equation}

With this notation, we now answer the question asked above,
that dealt with the fact of finding the value of $\theta$ for which
the asymptotic of the Green functions~(\ref{sp4-final_formulation_G_i_j})
will be the most easily calculated: we choose $\theta=\arg(\rho_{j/i})$---note that, from the definition of $z_{0}$ and (\ref{sp4-def_rho_j/i}),
we immediately obtain that
$\arg(\rho_{j/i})\in[\pi-\pi/n,\pi]$---, and the decomposition of
the contour $\Lambda(\theta,\theta)$ is
     \begin{equation*}
          \Lambda\left(\arg(\rho_{j/i}),\arg(\rho_{j/i})\right)=
          \left(\rho_{j/i}/|\rho_{j/i}|\right)\left[0,\epsilon\right]
          \cup \left(\rho_{j/i}/|\rho_{j/i}|\right)\left]\epsilon,1/\epsilon\right[
          \cup \left(\rho_{j/i}/|\rho_{j/i}|\right)\left[1/\epsilon,\infty\right].
     \end{equation*}
     According to this decomposition and to~(\ref{sp4-final_formulation_G_i_j}),
we consider now $G_{i,j}$ as the sum of three terms and we are going to
study successively the contribution of each of these terms.

\medskip

\medskip

{\raggedright \textit{\textbf{Contribution of the neighborhood of $0$.}}}
In order to evaluate the asymptotic of the integral (\ref{sp4-final_formulation_G_i_j})
on the contour $(\rho_{j/i}/|\rho_{j/i}|)[0,\epsilon]$, we are going to use the
expansion at $0$ of the function
%$(1/z)\sum_{w\in W_{n}}(-1)^{l(w)}x^{i_{0}}y^{j_{0}}(w(z))$.
     \begin{equation*}
          \frac{1}{z}\sum_{w\in W_{n}}\left(-1\right)^{l(w)}
          x^{i_{0}}y^{j_{0}}\left(w\left(z\right)\right).
     \end{equation*}
This is why we begin here by studying the asymptotic of the following integral:
     \begin{equation}
     \label{sp4-integral_general_p}
          \int_{\left(\rho_{j/i}/|\rho_{j/i}|\right)[0,\epsilon]}
          \frac{z^{k}}{x\left(z\right)^{i}
          y\left(z\right)^{j}}\text{d}z,
     \end{equation}
$k$ being some non-negative integer.
Using the equality $1/[x(z)^{i}y(z)^{j}]=\exp(-i\chi_{j/i}(z))$
as well as the expansion~(\ref{sp4-expansion_chi_0})
of $\chi_{j/i}$ at $0$ and then making the change of variable
$z=\rho_{j/i} t$, we obtain that~(\ref{sp4-integral_general_p}) is equal to
     \begin{equation}
     \label{sp4-integral_after_expansion}
          \rho_{j/i}^{k+1}\int_{0}^{\epsilon/|\rho_{j/i}|}
          t^{k}\exp\left(-i t\right)
          \exp\left(-i\sum_{p\geq 1}\nu_{2p+1}\left(j/i\right)
          \left(\rho_{j/i}t\right)^{2p+1}\right)
          \text{d}t.
     \end{equation}

But thanks to~(\ref{sp4-expansion_chi_0}), $|\nu_{2p+1}(j/i)|\leq 4(M+1)$
and therefore for all $t\in[0,\epsilon/|\rho_{j/i}|]$, we have
$|-i\sum_{p=1}^{\infty}\nu_{2p+1}(j/i)(\rho_{j/i}t)^{2p+1}|
\leq i\epsilon^{3}4(M+1)/(1-\epsilon^{2})$. This is why
     \begin{equation*}
          \exp\left(-i\sum_{p\geq 1}\nu_{2p+1}\left(j/i\right)
          \left(\rho_{j/i}t\right)^{2p+1}\right)=
          1+O\left(i\epsilon^{3}\right),
     \end{equation*}
the $O$ above being independent of $j/i\in[0,M]$ and of $t\in[0,\epsilon/|\rho_{j/i}|]$.
The integral~(\ref{sp4-integral_after_expansion}) can thus be calculated as
     \begin{equation*}
          \rho_{j/i}^{k+1}
          \left[1\hspace{-0.5mm}+\hspace{-0.5mm}O\left(i\epsilon^{3}\right)\right]
          \int_{0}^{\epsilon/|\rho_{j/i}|}
          t^{k}\exp(-i t)\text{d}t\hspace{-0.5mm}=\hspace{-0.5mm}
          \left(\rho_{j/i}/i\right)^{k+1}
          \left[1\hspace{-0.5mm}+\hspace{-0.5mm}O\left(i\epsilon^{3}\right)\right]
          \int_{0}^{i\epsilon/|\rho_{j/i}|}t^{k}\exp\left(-t\right)\text{d}t.
     \end{equation*}
In the sequel, we choose $\epsilon=1/i^{3/4}$, so that
$i\epsilon/|\rho_{j/i}|\to \infty$ and
$O(i\epsilon^{3})=O(1/i^{5/4})$.

One could be surprised by this choice of $\epsilon$;
in fact, in the upcoming paragraph~``Conclusion'',
we will see that
in order to obtain the asymptotic of the Green functions along the paths
of states $(i,j)\in\mathbb{Z}_{+}^{2}$ such that $j/i\to\tan(\gamma)\in]0,\infty[$, it would have been
sufficient~to have $O(i\epsilon^{3})=o(1)$, but for the paths $(i,j)\in\mathbb{Z}_{+}^{2}$ such that
$j/i\to 0$, it is necessary to have $O(i\epsilon^{3})=$ $o(1/i)$, what affords
the choice $\epsilon=1/i^{3/4}$.

Finally, we obtain that for this choice of $\epsilon$,
the integral (\ref{sp4-integral_general_p}) is equal to
     \begin{equation}
     \label{sp4-conclusion_integral_general_p}
          \int_{\left(\rho_{j/i}/|\rho_{j/i}|\right)[0,\epsilon]}
          \frac{z^{k}}{x\left(z\right)^{i}
          y\left(z\right)^{j}}\text{d}z=
          \left(\rho_{j/i}/i\right)^{k+1}k!
          \big[1+O\big(1/i^{5/4}\big)\big],
     \end{equation}
where the $O$ is independent of $j/i\in[0,M]$.

We are presently ready to obtain the asymptotic of
the integral~(\ref{sp4-final_formulation_G_i_j})
on the contour $(\rho_{j/i}/|\rho_{j/i}|)[0,\epsilon]$.
First, in accordance with~(\ref{sp4-s_s_expansion_x_i0_y_j0}),
we have that this integral equals
     \begin{equation*}
          \frac{1}{4\pi [p_{1,0} p_{1,-1}]^{1/2} }
          \sum_{p\geq 1} n \left[\kappa_{n p}
          \left(i_{0},j_{0}\right)-\overline{\kappa_{n p}}
          \left(i_{0},j_{0}\right)\right]
          \int_{\left(\rho_{j/i}/|\rho_{j/i}|\right)[0,\epsilon]}
          \frac{z^{n p - 1}}{x\left(z\right)^{i}
          y\left(z\right)^{j}}\text{d}z.
     \end{equation*}

Thus clearly, with~(\ref{sp4-conclusion_integral_general_p}), we
obtain that all the terms corresponding in the sum
above to $p\geq 2$ will be negligible w.r.t.\ the one associated with $p=1$.
In addition, by using the definition (\ref{sp4-def_f_n}) of the
harmonic function $f_{n}$ as well as (\ref{sp4-conclusion_integral_general_p}) for
$k=p n-1$ and $p\geq 1$, we get that the integral~(\ref{sp4-final_formulation_G_i_j})
on the contour $(\rho_{j/i}/|\rho_{j/i}|)[0,\epsilon]$ is equal to
     \begin{equation}
     \label{sp4-conclusion_contribution_neighborhood_0}
          \frac{1}{4\pi [p_{1,0} p_{1,-1}]^{1/2} }\left(-1\right)^{n}
          \left(n-1\right)!f_{n}\left(i_{0},j_{0}\right)
          \imath\left(\rho_{j/i}/i\right)^{n}
          \big[1+O\big(1/i^{5/4}\big)\big].
     \end{equation}

\medskip

\medskip

{\raggedright \textit{\textbf{Contribution of the neighborhood of $\infty$.}}}
The part of the contour close to $\infty$, namely
$(\rho_{j/i}/|\rho_{j/i}|)[1/\epsilon,\infty]$,
is related to the part $(\rho_{j/i}/|\rho_{j/i}|)[0,\epsilon]$
via the transformation $z\mapsto 1/\overline{z}$. Moreover,
it is clear from~(\ref{sp4-uniformization}) that for
$f=x$, $f=y$ or $f=\sum_{w\in W_{n}}(-1)^{l(w)}x^{i_{0}}y^{j_{0}}(w)$,
     \begin{center}
          $f\left(1/\overline{z}\right)=\overline{f\left(z\right)}.$
     \end{center}
Therefore, the change of variable $z\mapsto 1/\overline{z}$
immediately gives us that the contribution of the
integral~(\ref{sp4-final_formulation_G_i_j}) near $\infty$ is the
complex conjugate of its contribution near $0$.

\medskip

\medskip

{\raggedright \textit{\textbf{Contribution of the intermediate part.}}}
Let $A_{\epsilon}$ be the annular domain $\{z\in \mathbb{C}:
\epsilon\leq |z|\leq 1/\epsilon\}$. According to Figure~\ref{sp4-transformation_cycles},
for all $z\in \Lambda(\pi-\pi/n,\pi)\cap A_{\epsilon}$ we have
$|x(z)|>1+\eta_{x,\epsilon}$~and $|y(z)|>1+\eta_{y,\epsilon}$, where $\eta_{x,\epsilon}>0$ and $\eta_{y,\epsilon}>0$.
In fact, since $x'(0)\neq 0$ and $y'(0)\neq 0$, we can take
$\eta_{x,\epsilon}>\eta \epsilon$ and $\eta_{y,\epsilon}>\eta \epsilon$
for some $\eta>0$ independent of $\epsilon$ small enough.

Let us now consider
     \begin{equation*}
          L=\sup_{z\in \Lambda(\pi-\pi/n,\pi)}
          \left|\bigg[\sum_{w\in W_{n}}\left(-1\right)^{l\left(w\right)}
          x^{i_{0}}y^{j_{0}}\left(w\left(z\right)\right)\bigg]\bigg/
          \bigg[x^{i_{0}}y^{j_{0}}\left(z\right)\bigg]\right|,
     \end{equation*}
and let us show that $L$ is finite. For this, it is enough
to prove that the function $s$, defined by $s(z)=\big[\sum_{w\in W_{n}}(-1)^{l(w)}
x^{i_{0}}y^{j_{0}}(w(z))\big]\big/\big[x^{i_{0}}y^{j_{0}}(z)\big]$,
has no pole in $\Lambda(\pi-\pi/n,\pi)$---including $\infty$.
But by using~(\ref{sp4-uniformization}) and~(\ref{sp4-def_automorphisms_C_s}),
we see that the only poles of the numerator of $s$ are the
$z_{0}\exp(2\imath p \pi/n)$ for $p\in\{0,\ldots ,n-1\}$.
Among these $n$ points, only $z_{0}$ is in $\Lambda(\pi-\pi/n,\pi)$.
But in $s$, we have taken care of dividing by $x^{i_{0}}y^{j_{0}}(z)$,
so that $s$ is in fact holomorphic near $z_{0}$. Moreover,
it is easily shown that $s$ is holomorphic at $\infty$.
Finally, we have proven that $s$ has no pole in $\Lambda(\pi-\pi/n,\pi)$,
hence $s$ is bounded in $\Lambda(\pi-\pi/n,\pi)$; in other words, $L$ is finite.

The modulus of the contribution of~(\ref{sp4-final_formulation_G_i_j})
on the intermediate part $(\rho_{j/i}/|\rho_{j/i}|)]\epsilon,1/\epsilon[
\subset \Lambda(\pi-\pi/n,\pi)\cap A_{\epsilon}$
can therefore be bounded from above by
     \begin{equation}
     \label{sp4-upper_bound_intermediate_term}
          \frac{1}{4\pi [p_{1,0} p_{1,-1}]^{1/2} }
          \frac{1}{\epsilon^{2}}
          \frac{L}{(1+\eta\epsilon)^{i-i_{0}}
          (1+\eta\epsilon)^{j-j_{0}}}.
     \end{equation}

Note that the presence of the term $1/\epsilon^{2}$
in~(\ref{sp4-upper_bound_intermediate_term}) is due to the following: one $1/\epsilon$~appears
as an upper bound of the length of the contour, while the other $1/\epsilon$ comes
from an upper bound of the modulus of the term $1/z$ present in
the integrand of~(\ref{sp4-final_formulation_G_i_j}).

Then as before we take $\epsilon=1/i^{3/4}$, and then we use
the following straightforward upper bound, valid for $i$
large enough: $1/[1+\eta/i^{3/4}]^{i}\leq \exp(-[\eta/2]
i^{1/4})$. We finally obtain that for $i$ large enough,
(\ref{sp4-upper_bound_intermediate_term}) is equal to
$O(i^{3/2}\exp(-[\eta/2] i^{1/4}))$.

\medskip

\medskip

{\raggedright \textit{\textbf{Conclusion.}}}
We have seen that the contribution of the integral~(\ref{sp4-final_formulation_G_i_j})
in the neighborhood of $0$ is given by~(\ref{sp4-conclusion_contribution_neighborhood_0}),
that the contribution of~(\ref{sp4-final_formulation_G_i_j}) in the neighborhood of
$\infty$ is equal to the complex conjugate of~(\ref{sp4-conclusion_contribution_neighborhood_0})
and that the contribution of the residual part can be written as
$O(i^{3/2}\exp(-[\eta/2] i^{1/4})$. Therefore, with~(\ref{sp4-final_formulation_G_i_j})
and~(\ref{sp4-conclusion_contribution_neighborhood_0}), we obtain that
     \begin{equation}
     \label{sp4-btc}
          G_{i,j}=\frac{1}{4\pi [p_{1,0} p_{1,-1}]^{1/2} }\left(-1\right)^{n}
          \left(n-1\right)!f_{n}\left(i_{0},j_{0}\right)\imath
          \left[\left(\rho_{j/i}/i\right)^{n}-
          \left(\overline{\rho_{j/i}}/i\right)^{n}\right]
          +O\big(1/i^{n+5/4}\big).
     \end{equation}

Moreover, starting from~(\ref{sp4-def_rho_j/i}), we easily derive
     \begin{equation*}
          \left(\rho_{j/i}/i\right)^{n}-
          \left(\overline{\rho_{j/i}}/i\right)^{n}=
          \frac{2\imath\left(-1\right)^{n+1}}{4^{n}}
          \frac{
          \sin\big(n\arctan\big(\frac{j/i}{1+j/i}
          \tan\left(\pi/n\right)\big)\big)
          }
          {
          \big[\cos(\pi/n)^{2}\left(i^{2}
          +2 i j\right)+j^{2}\big]^{n/2}
          }.
     \end{equation*}

The latter equality,~(\ref{sp4-btc}) and Remark~\ref{sp4-rem_Nn} conclude
the proof of Theorem~\ref{sp4-Main_theorem_Green_functions} in the case of
$\gamma\in[0,\pi/2[$.

\begin{enumerate}

\item[$\ast$] Note that having
$o(1/i^{n})$ instead of $O(1/i^{n+5/4})$ would have been sufficient
for~$\gamma\in]0,\pi/2[$, since in this case, Remark~\ref{sp4-rem_Nn}
implies that $(\rho_{j/i}/i)^{n}-(\overline{\rho_{j/i}}/i)^{n}\sim K_{\gamma}/i^{n}$
with $K_{\gamma}\neq 0$.

\item[$\ast$] On the other hand, if $\gamma=0$ then
$(\rho_{j/i}/i)^{n}-(\overline{\rho_{j/i}}/i)^{n}\sim K_{0}j/i^{n+1}$
with $K_{0}\neq 0$ and it is necessary to have
something like $o(1/i^{n+1})$ in~(\ref{sp4-btc}), as it is actually the case with $O(1/i^{n+5/4})$.

\end{enumerate}

To prove Theorem~\ref{sp4-Main_theorem_Green_functions}
in the case $\gamma=\pi/2$, we would
consider $(i/j)\kappa_{j/i}$ rather than $\kappa_{j/i}$
and we would then use exactly the same analysis;
we omit the details.

\medskip

\medskip

{\raggedright \textit{\textbf{Acknowledgments.}}}
I would like to thank Irina Ignatiouk-Robert and Rodolphe~Garbit for
stimulating discussions we had together
concerning the topic of this article.
I also warmly thank Irina Kurkova who introduced me to this field of
research and also for her constant help and support during the elaboration
of this project.

\footnotesize

\bibliographystyle{plain}

\begin{thebibliography}{10}

\bibitem{MR513885}
A. Ancona.
\newblock Principe de {H}arnack \`a la fronti\`ere et th\'eor\`eme de {F}atou
  pour un op\'erateur elliptique dans un domaine lipschitzien.
\newblock {\em Ann. Inst. Fourier (Grenoble)} 28 (1978) 169--213.
MR513885

\bibitem{Smits}
R. Ba{\~n}uelos and R. Smits.
\newblock Brownian motion in cones.
\newblock {\em Probab. Theory Related Fields} 108 (1997) 299--319.
MR1465162

\bibitem{Bi1}
P. Biane.
\newblock Quantum random walk on the dual of {${\rm SU}(n)$}.
\newblock {\em Probab. Theory Related Fields} 89 (1991) 117--129.
MR1109477

\bibitem{Bi3}
P. Biane.
\newblock Minuscule weights and random walks on lattices.
\newblock In {\em Quantum probability \& related topics},
51--65. World Sci. Publ., River Edge, NJ, 1992.
MR1186654

\bibitem{Bou2}
N. Bourbaki.
\newblock {\em \'{E}l\'ements de math\'ematique}.
\newblock Fasc. XXXVIII: Groupes et alg{\`e}bres de Lie. Chapitre VII:
Sous-alg{\`e}bres de Cartan, {\'e}l{\'e}ments r{\'e}guliers. Chapitre VIII:
Alg{\`e}bres de Lie semi-simples d{\'e}ploy{\'e}es. Hermann, Paris, 1975.
MR0453824

\bibitem{BMM}
M.~Bousquet-M\'elou and M.~Mishna. Walks with small steps in the
quarter plane. In {\em Algorithmic Probability and Combinatorics},
1--40. Amer.\ Math.\ Soc., Providence, RI, 2010.


\bibitem{COL}
B. Collins.
\newblock Martin boundary theory of some quantum random walks.
\newblock {\em Ann. Inst. H. Poincar\'e Probab. Statist.} 40 
(2004) 367--384.
MR2060458

\bibitem{EJP2010-11}
D. Denisov and V. Wachtel.
\newblock Conditional limit theorems for ordered random walks.
\newblock {\em Electron. J. Probab.} 15 (2010)
292--322.
MR2609589

\bibitem{DO}
Y. Doumerc and N. O'Connell.
\newblock Exit problems associated with finite reflection groups.
\newblock {\em Probab. Theory Related Fields} 132 (2005) 501--538.
MR2198200

\bibitem{Dynkin}
E. Dynkin.
\newblock The boundary theory of {M}arkov processes (discrete case).
\newblock {\em Uspehi Mat. Nauk} 24 (1969) 3--42.
MR0245096

\bibitem{Dy62}
F. Dyson.
\newblock A {B}rownian-motion model for the eigenvalues of a random matrix.
\newblock {\em J. Mathematical Phys.} 3 (1962) 1191--1198.
MR0148397

\bibitem{KO}
P. Eichelsbacher and W. K{\"o}nig.
\newblock Ordered random walks.
\newblock {\em Electron. J. Probab.} 13 (2008) 1307--1336.
MR2430709

\bibitem{FIM}
G. Fayolle, R. Iasnogorodski and V. Malyshev.
\newblock {\em Random walks in the quarter-plane},
\newblock Springer-Verlag, Berlin, 1999.
MR1691900

\bibitem{III}
I. Ignatiouk-Robert.
\newblock Martin boundary of a killed random walk on $\mathbb{Z}_{+}^{d}$.
\newblock {\em Preprint} (2009).

\bibitem{II}
I. Ignatiouk-Robert.
\newblock Martin boundary of a reflected random walk on a half-space.
\newblock {\em Probab. Theory Related Fields}, to appear.

\bibitem{IL09}
I. Ignatiouk-Robert and C. Loree.
\newblock Martin boundary of a killed random walk on a quadrant.
\newblock {\em Ann. Probab.} 38 (2010) 1106--1142.
MR2674995

\bibitem{JS}
G. Jones and D. Singerman.
\newblock {\em Complex functions}.
\newblock Cambridge University Press, Cambridge, 1987.
MR0890746


\bibitem{Kestenn}
H. Kesten.
\newblock Hitting probabilities of random walks on {${\mathbb Z}^d$}.
\newblock {\em Stochastic Process. Appl.} 25 (1987) 165--184.
MR0915132

\bibitem{Komlos1}
J.~Koml{\'o}s, P.~Major and G.~Tusn{\'a}dy.
\newblock An approximation of partial sums of independent {${\rm RV}$}'s and
  the sample {${\rm DF}$}. {I}.
\newblock {\em Z. Wahrscheinlichkeitstheorie und Verw. Gebiete} 32 (1975) 111--131.
MR0375412

\bibitem{Komlos2}
J.~Koml{\'o}s, P.~Major and G.~Tusn{\'a}dy.
\newblock An approximation of partial sums of independent {RV}'s, and the
  sample {DF}. {II}.
\newblock {\em Z. Wahrscheinlichkeitstheorie und Verw. Gebiete} 34 (1976) 33--58.
MR0402883

\bibitem{KoS}
W. K{\"o}nig and P. Schmid.
\newblock Random walks conditioned to stay in Weyl chambers of type C and D.
\newblock {\em Electron. Comm. Probab.} 15 (2010) 286--296.

\bibitem{Lawlerloop}
M. Kozdron and G. Lawler.
\newblock Estimates of random walk exit probabilities and application to
  loop-erased random walk.
\newblock {\em Electron. J. Probab.} 10 (2005) 1442--1467.
MR2191635

\bibitem{KR}
I. Kurkova and K. Raschel.
\newblock Random walks in $\mathbb{Z}_+^2$ with non-zero drift absorbed at the
  axes.
\newblock {\em Bull. Soc. Math. France}, to appear.

\bibitem{LL}
G. Lawler and V. Limic.
\newblock The {B}eurling estimate for a class of random walks.
\newblock {\em Electron. J. Probab.} 9 (2004) 846--861.
MR2110020

\bibitem{LL2}
G. Lawler and V. Limic.
\newblock {\em Random Walk: A Modern Introduction}.
\newblock Cambridge University Press, Cambridge, 2010.
MR2677157

\bibitem{PW}
M. Picardello and W. Woess.
\newblock Martin boundaries of cartesian products of {M}arkov chains.
\newblock {\em Nagoya Math. J.} 128 (1992) 153--169.
MR1197035

\bibitem{mythesis}
K. Raschel.
\newblock {\em Chemins confin\'{e}s dans un quadrant}.
\newblock Th\`{e}se de doctorat de l'Universit\'{e} Pierre et Marie Curie,
  2010.

\bibitem{sl3}
K. Raschel.
\newblock Green functions and {M}artin compactification for killed random walks
  related to {S}{U}(3).
\newblock {\em Electron. Comm. Probab.} 15 (2010) 176--190.

\bibitem{EJP2010-37}
K. Uchiyama.
\newblock The {G}reen functions of two dimensional random walks killed on a
  line and their higher dimensional analogues.
\newblock {\em Electron. J. Probab.} 15 (2010) 1161--1189.
MR2659761

\bibitem{UchiyamaPreprint}
K. Uchiyama.
\newblock Random walks on the upper half plane.
\newblock {\em Preprint} (2010).

\end{thebibliography}

\end{document}